\documentclass[bj]{imsart}

\RequirePackage{amsthm,amsmath,amsfonts,amssymb}
\RequirePackage[numbers,sort]{natbib}

          \usepackage{enumerate}
          \usepackage{url}
          \usepackage{color}

\startlocaldefs

\numberwithin{equation}{section}
\theoremstyle{plain}
\newtheorem{Thm}{Theorem}[section]
\newtheorem{Lem}[Thm]{Lemma}
\newtheorem{Prop}[Thm]{Proposition}

\theoremstyle{remark}
\newtheorem{Rem}[Thm]{Remark}

\newcommand{\comment}[1]{}
\def\fddto{\xrightarrow{\textit{f.d.d.}}}
\newcommand{\ind}{{\bf 1}}

\def\inddd#1{{\ind}_{\left\{#1\right\}}}

\newcommand{\proba}{\mathbb P}

\newcommand{\esp}{{\mathbb E}}

\newcommand{\inv}{^{-1}}

\newcommand{\var}{{\rm{Var}}}

\newcommand{\eqnh}{\begin{eqnarray*}}
\newcommand{\eqne}{\end{eqnarray*}}
\newcommand{\eqnhn}{\begin{eqnarray}}
\newcommand{\eqnen}{\end{eqnarray}}
\newcommand{\equh}{\begin{equation}}
\newcommand{\eque}{\end{equation}}

\def\summ#1#2#3{\sum_{#1 = #2}^{#3}}
\def\prodd#1#2#3{\prod_{#1 = #2}^{#3}}
\def\sif#1#2{\sum_{#1=#2}^\infty}

\newcommand{\eqd}{\stackrel{d}{=}}

\def\topp#1{^{(#1)}}

\def\ccbb#1{\left\{#1\right\}}

\def\pp#1{\left(#1\right)}
\def\spp#1{(#1)}
\def\bb#1{\left[#1\right]}
\def\sbb#1{[#1]}
\def\mmid{\;\middle\vert\;}

\def\floor#1{\left\lfloor #1 \right\rfloor}

\def\vv#1{{\boldsymbol #1}}

\def\vvi{{\boldsymbol i}}

\def\vvt{{\boldsymbol t}}

\def\qmand{\quad\mbox{ and }\quad}

\def\qmwith{\quad\mbox{ with }\quad}
\def\mfa{\mbox{ for all }}

\def\mmas{\mbox{ as }}

\def\wt#1{\widetilde{#1}}
\def\wb#1{\overline{#1}}
\def\what#1{\widehat{#1}}

\def\limn{\lim_{n\to\infty}}

\def\limsupn{\limsup_{n\to\infty}}

\def\weakto{\Rightarrow}

\def\Z{{\mathbb Z}}

\def\R{{\mathbb R}}

\def\N{{\mathbb N}}

\def\calD{\mathcal D}

\def\calG{\mathcal G}

\def\calL{\mathcal L}
\def\calM{\mathcal M}

\def\calS{\mathcal S}

\def\topp#1{^{\scriptscriptstyle (#1)}}

\def\ddelta#1{\delta_{\pp{#1}}}
\def\PPP{{\rm PPP}}

\newcommand{\ConvP}{\overset{\proba}{\rightarrow}}
\newcommand{\EqD}{\overset{d}{=}}

\endlocaldefs

\begin{document}

\begin{frontmatter}
\title{Tail processes for stable-regenerative multiple-stable model}

\runtitle{Tail processes for stable-regenerative model}

\begin{aug}
\author[A]{\fnms{Shuyang} \snm{Bai}\ead[label=e1]{bsy9142@uga.edu}},
\and
\author[B]{\fnms{Yizao} \snm{Wang}\ead[label=e2]{yizao.wang@uc.edu}}

\address[A]{Department of Statistics, University of Georgia,
310 Herty Drive,
Athens, GA, 30602, USA. \printead{e1}}

\address[B]{Department of Mathematical Sciences, University of Cincinnati, 2815 Commons Way, Cincinnati, OH, 45221-0025, USA. \printead{e2}}
\end{aug}

\begin{abstract}
We investigate a family of discrete-time stationary processes defined  by multiple stable integrals  and renewal processes with infinite means. The model may exhibit behaviors of short-range or long-range dependence, respectively, depending on the parameters.  The main contribution is to establish a phase transition in terms of  the tail processes that characterize local clustering of extremes. Moreover, in the short-range dependence regime,  the model provides an example where the  extremal index is different from the candidate extremal index.
\end{abstract}



\end{frontmatter}

\section{Introduction and main results}\label{sec:1}
\subsection{The model and background}
The objective  of this paper is to study the local behavior of extremes of a family of stationary stochastic processes known as 
the
  stable-regenerative multiple-stable model that has attracted attention  in the studies of stochastic processes with {\em long-range dependence} \citep{samorodnitsky16stochastic,pipiras17long,beran13long}. 
In an accompanying paper \citep{bai21phase} the macroscopic/global limit of extremes are established in terms of convergence of the random sup-measures in the framework of \citet{obrien90stationary}, and a phase transition is revealed. Here, we characterize the microscopic/local limit of extremes, in terms of the   tail processes as introduced by  \citet{basrak09regularly}.  

The family of processes of our interest has a tail parameter $\alpha\in(0,2)$, a memory parameter $\beta\in(0,1)$, and a multiplicity parameter $p\in\N:=\{1,2,\ldots\}$. The representation is intrinsically related to renewal processes, 
 for which we introduce some notation. Consider a discrete-time renewal process  with the consecutive renewal times denoted by 
 $\vv\tau:=\{\tau_0,\tau_1,\dots\} \subset \N_0:=\{0,1,2,\ldots\}$.
  Here $\tau_0$ is the initial renewal time, and the inter-renewal times $(\tau_{i}-\tau_{i-1})_{i\ge 0}$ are i.i.d.\  $\N$-valued with cumulative distribution function $F$, that is, $F(x) = \proba(\tau_{i}-\tau_{i-1}\le x), i\in\N, x\ge 0$. Denote the probability mass function by $f(n) = \proba(\tau_{i}-\tau_{i-1} = n), n\in\N$. Throughout, we assume
\equh\label{eq:F}
\wb F(x) = 1-F(x) \sim \mathsf C_Fx^{-\beta} \mmas x\to\infty \qmwith \beta\in(0,1),
\eque
which  implies  an infinite mean,
and the following technical assumption
\equh\label{eq:Doney}
\sup_{n\in\N}\frac{nf(n)}{\wb F(n)}<\infty.
\eque
 By default, a renewal process starts at renewal at time 0, and hence $\tau_0 = 0$. 
 Note that our renewal processes may be {\em delayed}, that is, $\tau_0$  is not necessarily zero, and may be a random variable in $\N_0 = \N\cup\{0\}$, and we shall be specific when this is the case.
An important notion is the {\em stationary delay measure} of the renewal process, denoted by $\pi$. More precisely, $\pi$ is supported on $\N_0$ with
\[
\pi(k)\equiv \pi(\{k\}) = \wb F(k) = 1-F(k), \quad k\in\N_0.
\]
(For the sake of simplicity, we do not distinguish $\pi(k)$, the mass function at $k\in\N_0$, from $\pi(\{k\})$, the measure evaluated at the set $\{k\}$.) Note that the stationary delay  measure $\pi$ is a $\sigma$-finite and infinite measure on 
$\N_0$,
 since the renewal distribution has infinite mean. 
  More details are in Section \ref{sec:background}.
 
We then provide a series representation of the model.  
Consider
\equh\label{eq:series one-sided}
\sif i1\ddelta{y_i,d_i}\eqd \PPP\pp{(0,\infty]\times\N_0, \alpha x^{-\alpha-1}dxd\pi},
\eque
where the right-hand side is understood as the Poisson point process on $(0,\infty]\times \N_0$ with intensity measure $\alpha x^{-\alpha-1}dxd\pi$. 
In addition, let $\{\vv\tau\topp{i,d_i}\}_{i\in\N}$ denote, given the above Poisson point process, conditionally independent delayed renewal processes, each $\vv\tau\topp{i,d_i}$ with initial renewal time delayed at
$\tau_0 = d_i$ and   inter-renewal times following $F$.

In this paper,  we consider the {\em stable-regenerative multiple-stable model} given by
\equh\label{eq:series infty p}
\ccbb{X_k}_{k\in\N_0} \eqd\ccbb{\sum_{0<i_1<\cdots<i_p} [\varepsilon_\vvi][y_\vvi]\inddd{k\in \bigcap_{r=1}^p\vv\tau\topp{i_r,d_{i_r}}}}_{k\in\N_0},
\eque
where $\{\varepsilon_i\}_{i\in\N}$
  are i.i.d.\ Rademacher random variables independent of the point process in \eqref{eq:series one-sided}, $[\varepsilon_\vvi] = \varepsilon_{i_1}\times\cdots\times \varepsilon_{i_p}$, and similar notation applies for $[y_\vvi]$.  The name `stable-regenerative' comes from the fact that each renewal process $\vv\tau$, say non-delayed, has  a scaling limit as a 
  $\beta$-stable-regenerative
   random closed set  \citep{giacomin07random,bertoin99subordinators}. 
  The name `multiple-stable' comes from the fact that the series representation above corresponds to a  multiple stochastic integral with respect to a stable random measure.  
  See   \citep{bai20functional} for our model described explicitly in a multiple-stable-integral representation. 
    We shall work with the series representation only, so we omit the stochastic-integral representation here. The investigation of 
    multiple stochastic
     integrals has a long history, dated   from the celebrated work of \citet{ito51multiple} for the Gaussian case (a.k.a.~Wiener--It\^o integrals). See  \citep{krakowiak86random, szulga91multiple,kwapien92random,rosinski99product} for their extensions to non-Gaussian case, and \citep{samorodnitsky89asymptotic} for series representations. 
It is worth noting that the exclusion of the diagonals in the multiple stable integral concurs with the exclusion of the diagonals  in the multiple sum in \eqref{eq:series infty p}. 
Our model is actually a simplified version of the one investigated in \citep{bai20functional}; yet it preserves the key features and the renewal-process point of view facilitates our analysis.

When $p=1$, the  model $\{X_k\}_{k\in\N_0}$  in \eqref{eq:series infty p} is   a stationary  (non-Gaussian) stable process and was first introduced in \cite{rosinski96classes}. It exhibits non-standard asymptotic behaviors  in terms of limit theorems for sums and extremes as revealed in \citep{owada15functional,samorodnitsky19extremal}, and is hence viewed as a model with long-range dependence.       
For $p\ge 2$, 
a functional central limit theorem has been established recently in \citep{bai20functional}, under the assumption
 \equh\label{eq:beta_p}
 \beta_p:=p\beta -p+1\in(0,1),
 \eque
with a new family of  self-similar multiple-stable processes with stationary increments arising in the limit  (see also \citep{bai20limit,bai22limit} for   variations of the model \eqref{eq:series infty p} that scale  to  multiple-Gaussian   processes known as the Hermite processes). 
The functional central limit theorem  indicates that under the assumption \eqref{eq:beta_p}, the process exhibits behaviors of  long-range dependence.

Our motivation
 is to understand the limit behavior of extremes for all possible values of $\beta$ and $p$. 
Note that with $p=1$, necessarily $\beta_p = \beta\in(0,1)$,  and hence one only encounters  the regime of long-range dependence:  an extremal limit theorem in terms of random sup-measures \citep{obrien90stationary} has been established in \citep{samorodnitsky19extremal}, where the limit random sup-measures exhibit non-trivial dependence structure and in particular their marginal laws may no longer belong to the classical extreme-value distributions when $\beta> 1/2$. 
It was also shown that when $\beta_p<0$ the process  no longer exhibits  long-range dependence, with the limit random sup-measure being independently scattered \citep{bai21phase}.

Here, we examine the microscopic limit of extremes by investigating the tail processes for a full range of parameters. 
For a general stationary process with regularly varying tails $\{X_k\}_{k\in\N_0}$, its tail process characterizes the possible local clustering of extremes, and arises in a conditional limit theorem  as
 \equh\label{eq:Y}
 \calL\pp{\frac{X_0}{x},\dots,\frac{X_m}{x}\mmid |X_0|>x}\to \calL(Y_0,\dots,Y_m),  \mfa m\in\N,
 \eque
as $x\to\infty$. The left-hand side is understood as the conditional
joint
 law of $(X_i/x)_{i=0,\dots,m}$ given $|X_0|>x$, and the right-hand side the
 joint
  law of $(Y_0,\dots,Y_m)$.  If the above holds, then  $\{Y_i\}_{i\in\N_0}$ is referred to as the {\em tail process} of $\{X_i\}_{i\in\N_0}$. 
 The tail process was originally introduced by \citet{basrak09regularly}. Some closely related ideas date back to at least \citet{davis95point,davis98sample} (and see references therein for some closely related earlier developments) in the characterization of local clustering of extremes via point-process convergence.
See also \citet{kulik20heavy} for the state of the art on the topic.

There are several equivalent characterizations of the tail process. In particular, equivalently to \eqref{eq:Y}, one has \[
\calL\pp{\frac{X_0}{|X_0|},\dots,\frac{X_m}{|X_0|}\mmid |X_0|>x}\to (\Theta_0,\dots,\Theta_m),  \mfa m\in\N,
\]
 and the process $\vv\Theta=\{\Theta_i\}_{i\in\N_0}$ is referred as the {\em spectral tail process}. In such a case, 
 \[
\ccbb{Y_i}_{i\in\N_0}\eqd \ccbb{V_\alpha\Theta_i}_{i\in\N_0},
 \]
 where $V_\alpha$ is an $\alpha$-Pareto random variable ($\proba(V_\alpha>x) = x^{-\alpha}, x\ge 1$) independent from $\{\Theta_i\}_{i\in\N_0}$. Furthermore, $\vv\Theta$, when it exists, can be uniquely extended to a $\Z$-indexed stochastic process.

\subsection{Main result: a phase transition}
We first describe the spectral tail processes that will appear in the limit. 
Let $\vv\Theta^*\equiv \{\Theta_k^*\}_{k\in\N_0}$ be a $\{0,1\}$-valued sequence defined as follows: let $\{\vv\tau\topp i\}_{i=1,\dots,p}$ denote i.i.d.~copies of a standard (non-delayed) renewal process, with  the inter-renewal distribution 
function
 $F$ as in \eqref{eq:F}, and consider
\equh\label{eq:Theta*}
\Theta_{k}^* :=\begin{cases}
1, & \mbox{ if }  k\in\vv\eta,\\
0, & \mbox{ otherwise},
\end{cases}
\quad k=0,1,\dots \qmwith\vv\eta :=\bigcap_{r=1}^p \vv\tau\topp r.
\eque
In particular, $\Theta_0^* = 1$ since $0\in\vv\tau\topp r, r=1,\dots,p$ by definition. 
Moreover, let $\varepsilon$  be a  Rademacher random variable independent from $\vv\Theta^*$. We set
\equh\label{eq:Theta}
\vv\Theta :=  \varepsilon\vv\Theta^* = \{\varepsilon\Theta_i^*\}_{i\in\N_0}.
\eque 
We emphasize that the law of the spectral tail process  is completely determined by $F$ and $p$.  In fact, the intersection $\vv\eta$ is  again a non-delayed (i.e., $0\in\vv\eta$) renewal process. Note that the larger $p$ is, the smaller/sparser the intersection set $\vv\eta$ becomes. 
In particular, 
 $\vv\eta$ is possibly terminating, namely, $\eta_1 = \infty$ with strictly positive probability; this is the case when the renewal distribution of $\vv\eta$  has a mass at infinity, and we write $\vv\eta = 
\{0,\eta_1,\dots,\eta_k\}$
 if the $(k+1)$-th renewal time 
 is the first time with value infinity.
The renewal process $\vv\eta$ is terminating if and only if $\beta_p<0$, where $\beta_p$ is as in \eqref{eq:beta_p}. A quick derivation can be found in Section \ref{sec:EI} below (see the discussions after \eqref{eq:qFp}).  
The main result of the paper is the following.
\begin{Thm}\label{thm:tail}
For all $m\in\N$,
\[
\calL\pp{{\frac{X_0}{|X_0|},\dots,\frac{X_m}{|X_0|}}\mmid |X_0|>x} \to
\calL\pp{\varepsilon\Theta_0^*,\dots,\varepsilon\Theta_m^*},
\]
as $x\to\infty$, with the right-hand side as introduced in \eqref{eq:Theta}. 
\end{Thm}

\begin{Rem}
Theorem \ref{thm:tail} complements   our  results in \citep{bai21phase} so that now we have a complete picture regarding limit extremes at both macroscopic and microscopit levels,  as summarized in Table \ref{table:1}.
In the table, the extremal index (EI) and the candidate extremal index are two well-known notions in extreme-value theory for stationary sequences. See more background and discussions in Section \ref{sec:EI} below.
\begin{table}[ht!]
\begin{tabular}{|c|c|c|c|c|c}
\hline
regime 
& tail process & limit random sup-measure & candidate EI & EI\\
& (microscopic) & (macroscopic) & $\vartheta$ & $\theta$ \\
\hline
  super-critical, $\beta_p>0$ & non-terminating & with long-range dependence & 0 & 0 \\
   critical, $\beta_p=0$ &  non-terminating & independently scattered & 0 & 0 \\
    sub-critical, $\beta_p < 0$ & terminating & independently scattered & $\mathfrak q_{F,p}$ & $\mathsf D_{\beta,p}\mathfrak q_{F,p}$\\
    \hline
\end{tabular}
\vspace{.2cm}
\caption{Summary of phase transition. See $\mathfrak q_{F,p}$ and $D_{\beta,p}$ in \eqref{eq:qFp} and \eqref{eq:shape D} below.}\label{table:1}
\end{table}

 At microscopic level here, Theorem \ref{thm:tail} reveals that the super-critical ($\beta_p>0$) and critical ($\beta_p = 0$) regimes have the same type of asymptotic behavior, in the sense that the tail processes are not terminating; while at macroscopic level as described in \citep{bai21phase}, the critical ($\beta_p = 0$) and sub-critical ($\beta_p  < 0$) regimes have the same type of asymptotic behavior, with independently scattered $\alpha$-Fr\'echet random sup-measures arising in the limit. On the other hand, 
 the limit random sup-measures in the regime $\beta_p>0$ are of a new type, extending the one characterized in \citep{samorodnitsky19extremal}. They again exhibit long-range dependence and, unless $p = 1, \beta<1/2$, their marginal law goes beyond the family of classical extreme-value distributions due to an aggregation effect (see \citep{wang22choquet} for an explanation when $p=1$).

So, 
 the convergence of tail processes reveals a more delicate picture of local behaviors when $\beta_p\le 0$. 
At the same time, the tail process at the super-critical regime $\beta_p>0$ is still of a local nature, as it is impossible to recover from the tail processes the  random sup-measures that arise in the macroscopic limit obtained in \citep{bai21phase}. It is also remarkable that in the critical regime, while the macroscopic limit is independently scattered, the same as in the sub-critical regime, the normalization is {\em not the same}; this is related to the fact that the tail process in the critical regime is again non-terminating, reflecting the infinite size of the local cluster of extremes. 
\end{Rem}

\subsection{A notable example: when the candidate extremal index differs from the extremal index}\label{sec:EI}
A widely investigated notion for regularly-varying stationary stochastic process is the {\em extremal index}. A closely related notion is the {\em candidate extremal index}. For both,  our reference is again \citet{kulik20heavy}. 
We first recall the definitions when examining the stable-regenerative model in the sub-critical regime. 
Throughout this subsection, we assume  $\beta_p<0$.
In this case, the spectral tail process $\vv \Theta$ is terminating. 
Then  the {\em candidate extremal index}  $\vartheta$ is  (writing $a_+ = \max(a,0)$ for $a\in\R$)
\equh\label{eq:our var}
\vartheta  := \esp\pp{\sup_{i\ge 0}(\Theta_i)_+^\alpha - \sup_{i\ge 1}(\Theta_i)_+^\alpha   \mmid \Theta_0>0 } = \mathfrak q_{F,p},
\eque
with
\equh\label{eq:qFp}
\mathfrak q_{F,p} := \proba(\eta_1 = \infty) = 
\left(\sif n0 u(n)^p\right)^{-1} = 
\limn\wb F_p(n)\in(0,1),
\eque
where $u(n) = \proba(n\in\vv\tau)$ (see \eqref{eq:u} below).  
(See Remark \ref{rem:RT} on our convention of candidate extremal indices.)
To compute $\vartheta = \mathfrak q_{F,p}$, it is key to observe that 
$(\Theta_0)_+=\varepsilon_+$  and $(\Theta_i)_+= \varepsilon_+ \Theta_i^* $ is $\{0,1\}$-valued for all $i\ge 1$, and that $\varepsilon$ and $\vv\Theta$ are independent, whence
$\vartheta   = \proba\pp{\Theta_i^* = 0, \mfa i\in\N \ |\  \varepsilon=1}=\proba\pp{\Theta_i^* = 0, \mfa i\in\N } = \proba(\eta_1 = \infty)$. 
At the same time, the sequence $\{\Theta_i^*\}_{i\in\N_0}$ is all zeros except a finite number, 
say $\mathfrak G$, of ones, and because of a
 strong Markov property, $\mathfrak G$ is then geometric with 
\equh\label{eq:G}
\proba(\mathfrak G = k) = \proba(\eta_1 = \infty)\proba(\eta_1 < \infty)^{k-1}, k\in\N,
\eque and $\esp \mathfrak G = \mathfrak q_{F,p}\inv$. At the same time, 
$\esp \mathfrak G = \sif n0 \proba(n\in\vv\eta) = \sif n0\proba(n\in\vv\tau)^p\in (0,\infty)$, whence \eqref{eq:qFp} follows. 

Next, the 
{\em extremal index}
 $\theta$, if it exists, can be defined as the unique value from $[0,1]$ so that the following convergence holds
\equh\label{eq:EI def}
\proba\pp{\frac1{b_n}\max_{k=1,\dots,n}X_k \le x} = \exp\pp{-\theta x^{-\alpha}},
\eque
where $b_n$ is such that (see \eqref{eq:top_dominates} below)
\equh\label{eq:b_n}
\limn n\proba(X_1>b_n) = 1, \quad \mbox{ and in fact }\quad b_n \sim \pp{\frac12\frac{n\log^{p-1}n}{p!(p-1)!}}^{1/\alpha} 
\eque 
as $n\rightarrow\infty$.
For our model, it has been proved in \citep{bai21phase} that
\[
\theta = \mathsf D_{\beta,p}\mathfrak q_{F,p},
\]
with
\equh\label{eq:shape D}
\mathsf D_{\beta,p} := \sum_{s=q_{\beta,p}}^p\binom ps(-1)^{p-s}(-\beta_p)^{p-1} \qmwith q_{\beta,p} := \min\{q\in\N:\beta_q<0\}.
\eque
This follows from the convergence of random sup-measure established in \citep{bai21phase}, formally,
\equh\label{eq:RSM limit}
\frac1{b_n}\ccbb{\max_{k/n\in G}X_k}_{G\in\calG}\fddto (\mathsf D_{\beta,p}\mathfrak q_{\beta,p})^{1/\alpha}\ccbb{\calM_\alpha^{\rm is}(G)}_{G\in\calG},
\eque
where $\calG$ is the collection of all open intervals of $[0,1]$ and $\calM_\alpha^{\rm is}$ is the {\em independently scattered} $\alpha$-Fr\'echet random sup-measure. 
Indeed, \eqref{eq:EI def}  is the special case  of   marginal convergence with  $G=[0,1]$.
\begin{Rem}\label{rem:RT}
We have used a slightly different convention  compared to  \cite{kulik20heavy}. In fact, our candidate extremal index    $\vartheta$   
 is the same as the  \emph{right-tail} candidate extremal index in \citep[ Eq.~(7.5.4b)]{kulik20heavy}. The so-called candidate extremal index  in \citep{kulik20heavy}  is instead  given by (see \citep[ Eq.~(5.6.5)]{kulik20heavy}) 
\[
\vartheta^\circ = \esp\pp{\sup_{i\ge 0}|\Theta_i|^\alpha - \sup_{i\ge 1}|\Theta_i|^\alpha}. 
\] The right-tail index (our $\vartheta$) concerns only positive extreme values, while the  index $   \vartheta^\circ $ above concerns extreme values in absolute values. For our model, one readily checks that   $\vartheta$ and $\vartheta^\circ$ happen to coincide. 
\end{Rem}

It is clear from the definitions above that the extremal index $\theta$ and the candidate extremal index $\vartheta$ are characteristics of macroscopic and microscopic, respectively, behaviors of extremes. A priori these two values are not necessarily equal,
 although they can be shown to be equal 
   under an anti-clustering condition and a mixing condition  (see Remark \ref{rem:PP} below). However, we are not aware of any other examples of a regularly varying stochastic process so that $\theta$ and $\vartheta$
 are both between 0 and 1 and yet not the same.   
Therefore, our model is of special interest, and we  elaborate more the underlying mechanism for this rare phenomenon from the following two aspects.
 \begin{enumerate}[(i)]
 \item  First, we provide a simplified computation of the extremal index for the case $p=2$ in Section \ref{sec:p=2}. The limit theorem \eqref{eq:RSM limit}  above is established in \citep{bai21phase} by a different and much more involved proof.  The proof in in Section \ref{sec:p=2}, however, does not apply for $p\ge 3$. 
We hope the presentation here sheds light on the very unusual dependence structure of the model. 
\item 
Second, we prove that the so-called anti-clustering condition holds, that is, 
\equh\label{eq:AC}
\lim_{\ell\to\infty}\limsupn\proba\pp{\max_{\ell\le |k|\le r_n}|X_k|>b_n\eta\mmid |X_0|>b_n\eta} = 0, \mfa \eta>0,
\eque
for $r_n\rightarrow\infty$, $r_n = o(\log b_n)$,  in Section \ref{sec:anti}.
 (We actually prove a stronger version of it, known as the $\calS$ condition in the literature.)
 This, and combined with the fact that $\vartheta\ne\theta$, implies immediately that the commonly applied mixing-type condition in the classical approach fails for our model 
  at least for  the block size $r_n = o(\log b_n)$. 
See Remark \ref{rem:PP} below for more discussions. \end{enumerate}
\begin{Rem}\label{rem:PP}In the classical approach, in addition to the convergence of the tail processes and the verification of the anti-clustering condition \eqref{eq:AC}, if one could  also verify for our model the condition 
\equh\label{eq:mixing}
\limn \esp\exp\pp{-\summ i1n f(X_i/b_n)} - \pp{\esp\exp\pp{-\summ i1{r_n}f(X_i/b_n)}}^{\floor{n/r_n}} = 0,
\eque
with $b_n,r_n$ as in \eqref{eq:b_n}, then it follows that
\equh\label{eq:PP}
\summ i1n \ddelta{X_i/b_n,i/n}\inddd{X_i>0}\weakto \sif i1\mathfrak G_i\ddelta{\vartheta^{1/\alpha}\Gamma_i^{-1/\alpha},U_i} 
\eque
in the space of $\mathfrak M_p((0,\infty]\times[0,1])$, where $\{\Gamma_i\}_{i\in\N}$ are consecutive arrival times of a standard Poisson process, $\{U_i\}_{i\in\N}$ are i.i.d.~uniform random variables over $(0,1)$, $\{\mathfrak G_i\}_{i\in\N}$ are i.i.d.~copies of $\mathfrak G$ in \eqref{eq:G}, and all families are independent
(see \citep[Corollary 7.3.4]{kulik20heavy}). 
The convergence \eqref{eq:PP} would imply the extremal index $\theta = \vartheta$, whence a contradiction. Thus, in the context of our model, the relation \eqref{eq:mixing} fails to hold for $r_n\rightarrow\infty$ and $r_n=o(\log b_n)$. 

The idea of the classical approach is as follows. Within each block of size $r_n$, the local asymptotics is fully characterized by the tail processes, and different blocks behave asymptotically independently due to the condition \eqref{eq:mixing}. 
Usually, the condition \eqref{eq:mixing} follows from certain strong mixing properties 
(e.g. $\beta$-mixing; see \citep[Section 7.4.1]{kulik20heavy}).  
Our results indicate that our model does not enjoy very strong mixing properties. 
Nevertheless, we expect to be able to prove \eqref{eq:PP}  with $\vartheta^{1/\alpha}$ on the right-hand side replaced by $\theta^{1/\alpha}$. 
This is left for an upcoming work.
\end{Rem}

We conclude the introduction with 
a few more
 remarks.
\begin{Rem}
Note that $\theta = \mathsf D_{\beta,p}\vartheta$, and here we collect a couple of facts on $\mathsf D_{\beta,p}$:
\begin{enumerate}[(i)]
\item for all $\beta\in(0,1)$ so that $\beta_p<0$, $\mathsf D_{\beta,p}\in(0,1)$,
\item for all $\beta\in(0,1/2)$, $\mathsf D_{\beta,p} = 1-p\beta^{p-1}$. In particular, $\lim_{\beta\downarrow 0}\mathsf D_{\beta,p} = 1$. 
\end{enumerate} 
To see the first, introduce
\[
f_p(x) = \frac1{(p-1)!}\summ s0p \binom ps (-1)^{p-s}(s-x)_+^{p-1} = \frac1{(p-1)!}\summ s0p(-1)^{p-s}\binom ps(x-s)_+^{p-1},
\]
which is the 
probability density function of the
 so-called Irwin--Hall distribution, the one for the sum of $n$ i.i.d.~uniform random variables over $(0,1)$. Then we can write
\[
\mathsf D_{\beta,p} = (p-1)!(1-\beta)^{p-1}f_p((1-\beta)\inv)>0, \mfa \beta\in(0,1-1/p),
\]
or exactly when $\beta_p<0$.
To show that $\mathsf D_{\beta,p}<1$, it is equivalent to show that $f_p(x) <((p-1)!)\inv x^{p-1}$ for all $x>1$. But, recall that $f_p(x)$ is the $p$-times convolution function of uniform density function over $(0,1)$ evaluated at $x$. while $\frac1{(p-1)!}x^{p-1}$ is the $p$-times convolution of the indicator function over $[0,\infty)$ evaluated at $x$. The desired relation now follows. 

To see the second, first recall that for any polynomial function $Q(s)$ with degree at most $p-1$, one has $\summ s0p(-1)^{p-s}\binom ps Q(s) = 0$ 
since this corresponds to a $p$-times differencing operation. Then, take $Q(s) = (-\beta_s)^{p-1} = (s(1-\beta)-1)^{p-1}$ here. For $\beta\in(0,1/2)$, $q_{\beta,p} = 2$, 
and
it then follows that
\[
\mathsf D_{\beta,p} = \summ s2p (-1)^{p-s}\binom ps(-\beta_s)^{p-1} = -(-1)^pQ(0)-(-1)^{p-1}pQ(1) = 1-p\beta^{p-1}.
\]
The  fact $\mathsf D_{\beta,p}\in (0,1)$, which implies $\theta<\vartheta$, is consistent with  \citet[Lemma 7.5.4]{kulik20heavy}. We note that although the lemma there is stated only for 
$\R_+$-valued processes,   the proof readily extends to real-valued processes with the  right-tail 
 (cf.\ Remark \ref{rem:RT}) candidate extremal index  $\vartheta$ in \eqref{eq:our var}.
\end{Rem}

\begin{Rem}
In the case $\beta_p = 0$, we have in the critical regime,  in place of \eqref{eq:RSM limit},
\[
\proba\pp{\frac1{\wt b_n}\max_{k=1,\dots,n}X_k \le x} = \exp\pp{-\wt \theta x^{-\alpha}} \qmwith \wt b_n = \pp{\frac{n(\log\log n)^{p-1}}{\log n}}^{1/\alpha},
\]
and the value $\wt\theta>0$ explicitly computed in \citep{bai21phase}.
 Since $b_n /\wt b_n\rightarrow \infty$ as $n\rightarrow\infty$ with $b_n$ as in \eqref{eq:b_n},
the convergence above implies that the extremal index $\theta$ defined   via \eqref{eq:EI def} is zero. 
One may argue that in the super-critical regime, the extremal index is again $\theta = 0$ in the same way. 
\end{Rem}

\begin{Rem}
Our example shows that the candidate extremal index should be viewed as a local statistic only. In this regard,  we recall that despite the fact that it is commonly interpreted as the reciprocal of the mean cluster size, it is not always the case. In particular, \citet{smith88counterexample} provided an example where the extremal index is less than 1, while there is no extremal clustering from  a point-process-convergence  perspective. As further elaborated in \citep[Example 14.4.5]{kulik20heavy},  in this example the candidate extremal index and extremal index are the same. 
\end{Rem}

{\em The paper is organized as follows.} Section \ref{sec:background} provides related background, notably on multiple-stable processes and renewal processes. Section \ref{sec:tail} proves Theorem \ref{thm:tail}. Section \ref{sec:anti} proves the anti-clustering condition and the convergence of cluster point process when $\beta_p<0$.  Section \ref{sec:p=2} provides a   computation of the extremal index with $p=2,\beta\in(0,1/2)$. 

\section{Preliminary results}\label{sec:background}
\subsection{Renewal processes with infinite mean}
Throughout, our references on renewal processes are  \citet[Appendix A.5]{giacomin07random} and \citet{bingham87regular}.
Besides the notation and properties of renewal processes in Section \ref{sec:1}, we shall also use the renewal mass function of a non-delayed  renewal process $\vv\tau$ as follows
\equh\label{eq:renewal mass}
u(k) := \proba(k\in \vv\tau),\ k\in\N_0.
\eque
 It is   known that the assumption   
\equh\label{eq:u}
u(k)\sim \frac{k^{\beta-1}}{\mathsf C_F\Gamma(\beta)\Gamma(1-\beta)},
\eque
 as $k\rightarrow\infty$ implies 
the assumption \eqref{eq:F}, and that
under the assumption \eqref{eq:Doney} the two are equivalent (see \citep{doney97onesided}). 

Note that $\vv\tau$ is null-recurrent, and   the stationary delay measure associated with the inter-renewal distribution $F$ can be taken as $\pi(k) = \wb F(k), k\in\N_0$ (more generally one may set $\pi(k) = C\wb F(k)$, and we choose the constant $C=1$ for simplicity). The meaning of stationary delay measure is  explained in the Section \ref{sec:two-sided}, where for the sake of completeness, we present a proof for  a system of {\em two-sided renewal processes} to be shift-invariant.

Next, we recall some properties of the intersected renewal processes $\vv\eta = \bigcap_{r=1}^p\vv\tau\topp r = 
\{0,\eta_1,\eta_2,\dots\}
$ as introduced in \eqref{eq:Theta*}.
We let $F_p$ denote the cumulative distribution function of $\eta_1$ and we shall need the asymptotics of $\wb F_p(x) = 1-F_p(x)$. 
In summary, we have the following:
\[
\wb F_p(n) \sim \begin{cases}
\displaystyle n^{-\beta_p}\frac{(\mathsf C_F\Gamma(\beta)\Gamma(1-\beta))^p}{\Gamma(\beta_p)(\Gamma(1-\beta_p))}, & \mbox{ if } \beta_p>0; \\\\
\displaystyle
\frac{(\mathsf C_F\Gamma( \beta)\Gamma(1-\beta))^p}{\log n}, & \mbox{ if } \beta_p = 0.
\end{cases}
\]
When  $\beta_p<0$, the renewal process is terminating, with the probability $\mathfrak q_{F,p} = \proba(\eta_1 = \infty)$ as given in \eqref{eq:qFp}. 
For more details, see \citep{giacomin07random,bingham87regular} and  \citep{bai21phase}.
\subsection{An equivalent series representation}

Here, we explain a different representation that shall be needed for our proofs, for finite-dimensional distributions of the process.
We are interested in the joint law of $(X_0,\dots,X_m)$, for some $m$ fixed. Then, for all those $i\in\N_0$ such that 
$\vv\tau\topp{i,d_i}\cap\{0,\dots,m\} = \emptyset$, they do not contribute in the series representation. 
Therefore, by a standard thinning argument of Poisson point processes, it follows that
\equh\label{eq:thinning}
\sif i1\ddelta{\varepsilon_i\Gamma_i^{-1/\alpha},\ \vv\tau\topp{i,d_i}\cap\{0,\dots,m\}}\inddd{\vv\tau\topp{i,d_i}\cap\{0,\dots,m\}\ne\emptyset}\eqd\sif i1\ddelta{w_m^{1/\alpha}\varepsilon_i\Gamma_i^{-1/\alpha},\ R_{m,i}},
\eque
where on the right-hand side, 
\[
w_m = \summ k0{m} \pi(k) =  \summ k0{m} \wb F(k) \sim \frac{\mathsf C_F}{1-\beta}\cdot m^{1-\beta},
\]
the random variables $\{\Gamma_i\}_{i\in\N}$ are consecutive arrival times of a standard Poisson process, $\{\varepsilon_i\}_{i\in\N}$ are i.i.d.~Rademacher random variables, $\{R_{m,i}\}_{i\in\N}$ are i.i.d.~random closed subsets of $\{0,\dots,m\}$ with the law $R_{m,i}\eqd R_m$ described below, and all families are independent. 
Suppose $\vv\tau^*$ is a delayed renewal process with the stationary delay measure $\pi$ and renewal distribution $F$ defined on a   measurable space  with respect to  an infinite measure $\mu^*$ (since $\pi$ is an infinite measure). Then, one can introduce a probability measure $\mu_m$ on the same measurable space via
\[
\frac{d\mu_m}{d\mu^*} = \frac{\inddd{\vv\tau^*\cap\{0,\dots,m\} \ne\emptyset}}{\mu^*(\{\vv\tau^*:\vv\tau^*\cap\{0,\dots,m\}\ne\emptyset\})} = \frac{\inddd{\vv\tau^*\cap\{0,\dots,m\} \ne\emptyset}}{w_m}.
\]
Then, the law of $R_m$ is the one induced by $\vv\tau^*$ with respect to the probability measure $\mu_m$. 
Moreover, it is immediately verified that
 \begin{enumerate}[(i)]
 \item $\proba(k\in R_m) = 1/w_m, k=0,\dots,m$ (shift invariance).
  \item $\proba(\min(R_m\cap\{k+1,k+2,\dots\}) \le k+ j\mid k\in R_m) = F(j)$ (Markov/renewal property).
 \end{enumerate}
By \eqref{eq:thinning}, we work with the following equivalent representation of \eqref{eq:series infty p}: 
\equh\label{eq:p>=1}
\ccbb{X_k}_{k=0,\dots,m}\eqd \ccbb{w_m^{p/\alpha}\sum_{0<i_1<\cdots<i_p}\frac{[\varepsilon_\vvi]}{[\Gamma_\vvi]^{1/\alpha}}\inddd{k\in \bigcap_{r=1}^p R_{m,i_r}}}_{k=0,\dots,m},
\eque
where on the right-hand side the notation  is as in \eqref{eq:thinning}.

\begin{Rem}
Note that the representation in \eqref{eq:p>=1} is slightly different from the one used in \citep{bai21phase}: we include the time zero here, which is more convenient when studying the tail processes. Such a change does not effect   the normalization or the limit object. 
\end{Rem}

\subsection{A two-sided representation}\label{sec:two-sided}

We provide another series representation of $\{X_k\}_{k\in\Z}$, based on two-sided renewal processes. We do not need this in the rest of the paper. Such a representation is of its own interest, and helps illuminate the notion of delayed stationary distribution which is now defined in a two-sided manner. The results in this section may have been known in the literature but we were unable to find it.

Recall \eqref{eq:series one-sided} and \eqref{eq:series infty p}. In place of $\vv\tau\topp{i,d_i}$ now we introduce
$\vv\tau\topp{i,d_i,g_i}$: this is a two-sided renewal processes to be introduced with the first renewal time to the right of the origin (included) is $d_i$, the first renewal time to the left of the origin (not included) is $-g_i$ (so $g_i\in\N$).

We first introduce
\[
\sif i1\ddelta{y_i,d_i,g_i} \eqd \PPP((0,\infty)\times\N_0\times\N,\alpha x^{-\alpha-1}dxd\wt \pi),
\]
with the measure $\wt\pi$ on $\N_0\times\N$ determined by the following mass function
\[
\wt \pi(d,g):= \pi(d)\cdot \frac{f(d+g)}{\wb F(d)} = f(d+g), \ d\in\N_0, g\in\N.
\]
Note that the factor $f(d+g)/\wb F(d)$ is the probability mass function at $d+g, g\in\N$, of the conditional law of a renewal time with respect to $F$, given that the renewal time is strictly larger than $d$.

Now we attach independent renewal processes to each pair of $(d_i,g_i)$.  Let each ${\vv\tau}\topp{i,\rightarrow}$ be a copy of $\vv\tau = 
\{\tau_0,\tau_1,\dots\}$,
 a renewal process starting from the origin (so $\tau_0 = 0$) with renewal distribution $F$: ${\vv\tau}\topp{i,\rightarrow} \eqd 
 \{0,\tau_1,\tau_2,\dots\}$.
   Similarly, let each ${\vv\tau}\topp {i,\leftarrow}$ be a time-reversed renewal process, starting from zero: so with $\vv\tau$ as before, ${\vv\tau}\topp{i,\leftarrow}\eqd 
   \{0,-\tau_1,-\tau_2,\dots\}$. All the renewal processes are assumed independent from everything else. 
Then, we set
\equh\label{eq:tau two-sided}
\vv\tau\topp{i,d_i,g_i} := (d_i+{\vv\tau}\topp{i,\rightarrow})  \cup (-g_i+{\vv\tau}\topp{i,\leftarrow}), 
\eque
where $d_i+{\vv\tau}\topp{i,\rightarrow},-g_i+{\vv\tau}\topp{i,\leftarrow}$ are understood as   subsets  of $\Z$. 
Each two-sided renewal   $\vv\tau\topp{i,d_i,g_i}$ takes value in the path space  
\[
S = \ccbb{\vv t = \{t_i\}_{i\in\N}: t_i\in\Z, \mbox{ all distinct}}
\]
equipped with the cylindrical $\sigma$-field, i.e., the $\sigma$-field generated by  $\{\vv t\in S:\ k\in \vv t\}$, $k\in \Z$.
 
The two-sided series representation is then given by
\equh\label{eq:two-sided}
\ccbb{X_k}_{k\in\Z} \eqd \ccbb{\sum_{\vvi\in\calD_p}[\varepsilon_\vvi][y_\vvi]\inddd{k\in \bigcap_{r=1}^p\vv\tau\topp{i_r,d_{i_r},g_{i_r}}}}_{k\in\Z}  
\eque
with  $\{\varepsilon_i\}$  as in  \eqref{eq:series infty p} which is independent of everything else, and \[\calD_p=\{(i_1,\ldots,i_p)\in \N^p:\ i_1<\ldots<i_p \}.\]
\begin{Lem}
The two-sided representation \eqref{eq:two-sided} represents a stationary process $\{X_k\}_{k\in\Z}$ which restricted to $k\in\N_0$ has the representation \eqref{eq:series infty p}. 
\end{Lem}
\begin{proof}
It is known that the multiple series in \eqref{eq:two-sided}
converges  almost surely and unconditionally 
(e.g.,
\citep[Theorem 1.3 and Remark 1.5]{samorodnitsky89asymptotic}).
By construction, \eqref{eq:two-sided} restricted to $k\in\N_0$ is the same as \eqref{eq:series infty p}. 
It remains to prove stationarity of the process $\{X_k\}_{k\in\Z}$ 
This shall follow from a shift-invariance property of the intensity measure of the Poisson point process 
\[
\sif i1\ddelta{y_i,\vv\tau\topp{i,d_i,g_i}}. 
\]
(We omit the discussions regarding $\{\varepsilon_i\}_{i\in\N}$ for the sake of simplicity.)
Introduce the following shift operation 
  $\mathsf B$ on $\vvt\in S$: $\mathsf B(\vvt) = \vvt +1$.
for $\vvt = \{t_i\}_{i\in\N}\in S$, $\mathsf B(\vvt) = \vvt+1 = \{t_i+1\}_{i\in\Z}$ (so that   $d+{\vv\tau}\topp{i,\rightarrow} = \mathsf B^d({\vv\tau}\topp{i,\rightarrow})$).
Let $P_{\vv\tau}$ denote the probability measure on $S$ induced by the distribution of  the  renewal process $\vv\tau$ as above. 
Then, the Poisson point process above is on the space $(0,\infty)\times S$, with  its intensity measure denoted by $\alpha x^{-\alpha-1}dx dQ$, and $Q$ is defined as the push-forward measure   with respect to the construction \eqref{eq:tau two-sided} from the product measure $\wt \pi\times P_{\vv\tau}\times P_{\vv\tau}$ on $(\N_0\times\N) \times S\times S$.

Now, the stationarity of the process $\{X_k\}_{k\in\Z}$ follows from the following shift-invariance property for the measure $Q$ on $S$:
\equh\label{eq:Q invariance}
  Q\circ \mathsf B^{-1}  = Q.
\eque
To prove the above, we start by writing for every $\vvt\in S$ and $k\in \Z$  that
\[
\mathsf d_k(\vvt) := \min\{t_i:t_i\ge k\} - k, \qmand \mathsf g_k(\vvt) :=  k-\max\{t_i:t_i<k\}.
\]
In words, $\mathsf d_k(\vvt)$ represents the distance between $k$ and the first element to the right of $k$ (included) of $\vvt$, and $\mathsf g_k(\vvt)$ the distance between $k$ and the first element to the left of $k$ (not included).   
Our construction tells that
\equh\label{eq:Q}
Q(\mathsf d_0(\vv t) = d,\mathsf g_0(\vv t) = g)  = \wt \pi(d,g) = f(d+g) \mfa   d\in\N_0,\ g\in\N. 
\eque
To show \eqref{eq:Q invariance}, based on a decomposition with respect to $d_0(\vv t)$ and $\mathsf g_0(\vv t)$ and Dynkin's theorem, it suffices  to show    
\begin{equation} \label{eq:Q inv intermediate}
 Q  \left( \vv s  \subset   B^{-1}\vv t  ,\ \mathsf d_0(B^{-1}\vv t) =d,\ \mathsf g_0(B^{-1}\vv t) = g \right)= Q\left(     \vv s \subset \vv t   ,\ \mathsf d_0(\vv t)=d, \ \mathsf g_0(\vv t) = g \right)
 \end{equation} 
for all $d\in\N_0$, $g\in\N$  and    $\vv s \in S$ of finite size such that  $\mathsf d_0(\vv s)=d$ and $\mathsf g_0(\vv s)=g$. Write $\vv s=\{ s_{-m},\ldots, s_{-1},s_0, \ldots, s_{n}     \}$,  such that  $s_{i}<s_{i+1}$ and $s_{-1}=-g$ and $s_0=d$, $m\in \N$, $n\in \N_0$.
   Then the left-hand side of \eqref{eq:Q inv intermediate} is equal to
\equh\label{eq:Q 2nd}
 Q  (  s_{i}+1 \in \vv t ,\ i=-m,\ldots,n ,\  \mathsf d_1(\vv  t)=d,\ \mathsf g_1( \vv  t) = g ).
\eque 
Now, if $d\ge 0$ and $g\ge 2$, then $\{\vv t\in S:\ \mathsf d_1(\vv  t)=d,\ \mathsf g_1(\vv  t)=g\}=\{\vv t\in S:\ \mathsf d_0(\vv  t)=d+1, \ \mathsf g_0(\vv  t)=g-1\}=\{\vv t\in S:\ \mathsf d_0(\vv  t)=s_{0}+1, \ \mathsf g_0(\vv  t)=s_{-1}+1\}$.   Hence  the expression in \eqref{eq:Q 2nd} in this case becomes 
\begin{equation}\label{eq:Q 3rd}
f(d+g) \prod_{-m\le i\le n-1, i\neq -1} u(s_{i+1}-s_i),
\end{equation}
where we have used   the relation \eqref{eq:Q}  as well as the renewal property on the two sides with  the renewal mass function   $u$ as in \eqref{eq:renewal mass}. Next suppose $d\ge 0$ and  $g = 1$. Set $\mathsf d'_0 (\vv t)=\min ( \vv t \cap \{d_0(\vv t)+1,d_0(\vv t)+2,\ldots\})$, namely, the second element of $\vv t$ to the right of the origin (included).  We have  
$
\{\vv t\in S:\ \mathsf d_1(\vv t) = d,\  \mathsf g_1(\vv t) = 1\}  =\{\vv t\in S:\  \mathsf d_0(\vv t) = 0,\  \mathsf d'_0 (\vv t)= d+1\}=\{\vv t\in S:\  \mathsf d_0(\vv t) = s_{-1}+1,\  \mathsf d'_0 (\vv t)= s_0+1\}$. 
Then by construction and the renewal property, the expression in \eqref{eq:Q 2nd} in this case becomes
 \[
 Q(\mathsf d_0(\vv t) = 0)\proba(\tau_1 = d+1) \prod_{-m\le i\le n-1, i\neq -1} u(s_{i+1}-s_i).
 \] 
Note that $Q(\mathsf d_0(\vv t) =0) = \pi(0) = \wb F(0) =  1$ and $\proba(\tau_1 = d+1)=f(d+1)$, and hence the formula above coincides with  \eqref{eq:Q 3rd} when $d\ge 0$ and $g=1$. Therefore, we have shown that the left-hand side of \eqref{eq:Q inv intermediate} is equal to the expression in \eqref{eq:Q 3rd} for all $d\in \N_0$ and $g\in \N$.   The proof is concluded once noticing that the right-hand side of \eqref{eq:Q inv intermediate} is readily \eqref{eq:Q 3rd} for all $d\in \N_0$ and $g\in \N$. 
\end{proof}
\section{Convergence for tail processes}\label{sec:tail}
We prove Theorem \ref{thm:tail}.  Below $C$ will denote a generic positive constant whose value may change from one expression to another.
 
In order to establish $(m+1)$-dimensional multivariate regular variation, we shall work with the representation \eqref{eq:p>=1} with $ m$ throughout fixed.
We introduce some notation. Set  
\begin{align}
\wt\ell_k(1) & :=\min\ccbb{i\in\N:k\in R_{m,i}},\nonumber\\
 \wt\ell_k(s) &:=\min\ccbb{i>\wt \ell_k(s-1):k\in R_{m,i}},\ s\ge 2, \ k=0,\dots,m,\label{eq:ell_k}
\end{align}
namely, $\wt\ell_k(1),\wt\ell_k(2),\ldots$ are the successive $i$-indices such that $k\in R_{m,i}$. For $\vv i=(i_1,\ldots,i_p) \in \calD_p$, we write $\wt\ell_k(\vv i)=(\wt\ell_k(i_1),\ldots, \wt\ell_k(i_p))$.
In this way, we write 
\equh\label{eq:top}
X_k = w_m^{p/\alpha}\sum_{\vvi\in\calD_p}\frac{\sbb{\varepsilon_{\wt\ell_k(\vvi)}}}{\sbb{\Gamma_{\wt\ell_k(\vvi)}}^{1/\alpha}} \qmand
T_k:=w_m^{p/\alpha}{\frac{\sbb{\varepsilon_{\wt\ell_k((1,\dots,p))}}}{\sbb{\Gamma_{\wt\ell_k((1,\dots,p))}}^{1/\alpha}}}, \quad  k=0,\dots,m.
\eque
Here and below the notational convention is to write the product  $\varepsilon_{\wt\ell_k(i_1)}\cdots\varepsilon_{\wt\ell_k(i_p)}$ as $[\varepsilon_{\wt\ell_k(\vvi)}]   $, and similarly for $[\Gamma_{\wt\ell_k(\vvi)}]$. Note that the indicator functions are dropped in the representation \eqref{eq:top}. 
Moreover, in order to study the marginal distribution, taking into account of  the thinning probability $\proba(k\in R_{m,i})=w_m^{-1}$, we shall work with the following representation for each $k\in \N_0$ (but not jointly in $k$):
\equh\label{eq:thinning Xk}
\pp{X_k,T_k} \eqd  \pp{\sum_{\vvi\in\calD_p}\frac{[\varepsilon_\vvi]}{[\Gamma_\vvi]^{1/\alpha}},\frac{\sbb{\varepsilon_{(1,\dots,p)}}}{\sbb{\Gamma_{(1,\dots,p)}}^{1/\alpha}}}.
\eque
Recall that from \citep{samorodnitsky89asymptotic}, we have  
\equh\label{eq:product_tail}
\mathsf q_p(x):=\proba\pp{\bb{\Gamma_{1:p}}^{-1}>x} \sim \frac{x^{-1}\log^{p-1}x}{p!(p-1)!}.
\eque
as $x\to\infty$. 
Here and below we write $\Gamma_{1:p} = \Gamma_{(1,\dots,p)} = \Gamma_1\times\cdots\times \Gamma_p$. 
Moreover we have the following.
\begin{Lem}\label{lem:1}
We have
$\proba\spp{\max_{(1,\dots,p)\ne\vvi\in\calD_p}\bb{\Gamma_{\vvi}}^{-1/\alpha}>x} = \proba\spp{(\Gamma_1\cdots \Gamma_{p-1}\Gamma_{p+1})^{-1/\alpha}>x}$
and
\equh\label{eq:single_product}
\limsup_{x\to\infty}\frac{\proba\pp{(\Gamma_1\cdots \Gamma_{p-1}\Gamma_{p+1})^{-1/\alpha}>x}}{\mathsf q_{p-1}(x^\alpha)} \le \esp \Gamma_2\inv<\infty.
\eque
Moreover,  as $x\rightarrow\infty$, 
\equh\label{eq:top_dominates}
\proba(|X_k|>x)\sim 2\proba( X_k >x) \sim  \proba(|T_k|>x) = \mathsf q_p(x^\alpha)\sim   \frac{\alpha^{p-1} x^{-\alpha} \log^{p-1}(x)}{p!(p-1)!},
\eque
and
\equh\label{eq:remainder_rate}
\proba(|X_k-T_k|>x)\le C\mathsf q_{p-1}(x^\alpha)=O(      x^{-\alpha} \log^{p-2}(x) ).
\eque

\end{Lem}
 In words, $T_k$ is the product of $\Gamma_i$ for the smallest $p$ $i$-indices   such that $k\in R_{m,i}$, and the tails of $ X_k $ are determined by this single term alone. Then, by Lemma \ref{lem:1},  the joint representation \eqref{eq:top} allows us  to replace the $X_k$'s by $T_k$'s and deal with the joint law of $T_0,\dots,T_m$. The above can be read from \citep{samorodnitsky89asymptotic}, which deals with a more general setup. For the sake of convenience we include a proof here.
\begin{proof}[Proof of Lemma \ref{lem:1}]
The first claim \eqref{eq:single_product}  follows from monotonicity. 
Let $\wt\Gamma_2\eqd\Gamma_2$ be a random variable independent from everything else. We then have
\[
(\Gamma_1\cdots \Gamma_{p-1}\Gamma_{p+1})^{-1/\alpha}  \eqd \bb{\Gamma_{1:p-1}}^{-1/\alpha}\cdot \pp{\Gamma_{p-1}+\wt\Gamma_2}^{-1/\alpha} \le \bb{\Gamma_{1:p-1}}^{-1/\alpha}\wt\Gamma_2^{-1/\alpha}.
\]
The right-hand side is now a product of two independent random variables. The second part now follows from Breiman's lemma (e.g.~\citep{resnick07heavy}).

Notice that in view of \eqref{eq:product_tail}, the relation \eqref{eq:remainder_rate} implies \eqref{eq:top_dominates}, and hence it is left to prove \eqref{eq:remainder_rate}. Let $M$ be a fixed integer such that $M\ge 2p/\alpha+p$.  
Note that the conclusion concerns only the marginal distribution, and hence we can work with the simplified representation of $X_k$ in \eqref{eq:thinning Xk}.
By \eqref{eq:single_product}, any finite sum of terms of the form $[\varepsilon_\vvi] [\Gamma_\vvi]^{-1/\alpha}$ with $\vvi\neq (1,\dots,p)$  will have lighter tails than $[\Gamma_{(1,\dots,p)}]^{-1/\alpha}$.  
Since the sets $\{\vvi\in\calD_q: i_q\le M\}$, $q=0,\ldots,p-1$ are finite (when $q=0$ the set is understood as an empty set),  it suffices to show the following:    for any fixed $i_1<\cdots<i_q\le M$ and $q=0,1,\ldots, p-1$,   
\begin{equation}\label{eq:remainder series}
\proba\left([\Gamma_{i_{1:q}}]^{-1/\alpha}  \left|\sum_{M<i_{q+1}<\cdots<i_p } \frac{[\varepsilon_{i_{q+1:p}}]}
{  [\Gamma_{i_{q+1:p}}]^{1/\alpha}}\right|>x \right)  =o\left( x^{-\alpha} \log ^{p-1}x \right),\quad x\rightarrow \infty,
\end{equation}
here and below we write $i_{1:q} = (i_1,\dots,i_q)$, $[\Gamma_{i_{1:q}}] = \Gamma_{i_1}\times\cdots\times \Gamma_{i_q}$, and similarly  for other terms.
If $q=0$, in view of the inequality $\esp\left( [\Gamma_{i_{ 1:p}}]^{-2/\alpha}\right)\le c (i_1\ldots i_p)^{-2/\alpha} $ which holds by the choice of $M$  (cf.\ \cite[Eq.(3.2)]{samorodnitsky89asymptotic}), as well as the orthogonality induced by $[\varepsilon_{i_{1:p}}]$, one can verify that the second moment of the multiple series in \eqref{eq:remainder series} is finite, and hence the tail decay is $O(x^{-2}) =o\left( x^{-\alpha} \log ^{p-1}x\right)$ as $x\rightarrow\infty$. Suppose $q>0$ below.
Then the probability above is bounded by
\equh\label{eq:P}
\proba\left([\Gamma_{(1,\dots,q)}]^{-1/\alpha}  \left|\sum_{M<i_{q+1}<\cdots<i_p} \frac{[\varepsilon_{i_{q+1:p}}]}{  [\Gamma_{i_{q+1:p}}]^{1/\alpha}}\right|>x \right) .
\eque
Using $[\Gamma_{i_{q+1:p}}]^{-1/\alpha}\le \prodd r{q+1}p(\Gamma_{i_r}-\Gamma_{q})^{-1/\alpha}$,
 conditioning on $\{\Gamma_i\}_{i\in\N}$ and applying the contraction principle for multilinear form of Rademacher random variables \cite[Theorem 3.6]{delapena94contraction}, the probability in \eqref{eq:P} is bounded,  up to a multiplicative constant, by
\begin{align*}
\proba\left([\Gamma_{(1,\dots,q)}]^{-1/\alpha}   \left|\sum_{M<i_{q+1}<\cdots<i_p} [\epsilon_{i_{q+1:p}}] \prodd r{q+1}p(\Gamma_{i_r}-\Gamma_{q})^{-1/\alpha}\right|>x \right).
\end{align*}
Note that $[\Gamma_{(1,\dots,q)}]^{1/\alpha}$ is independent of the absolute value part above, which can be shown to have a finite second moment similarly as before using the fact   $(\Gamma_{i}-\Gamma_{q})_{i>M}\eqd (\Gamma_{i})_{i>M-q}$ with the choice of $M$. Therefore, by Breiman's lemma again, the probability above is of the same order as $\proba([\Gamma_{(1,\dots,q)}]^{-1/\alpha}>x)\sim \mathsf q_{q}(x^\alpha)  =o(x^{-\alpha} \log ^{p-1}x)$ as $x\to\infty$ in view of \eqref{eq:product_tail}.
\end{proof}

\begin{proof}[Proof of Theorem \ref{thm:tail}]
 Write
\equh
T_m^*:=w_m^{p/\alpha}\frac{\sbb{\varepsilon_{(1,\dots,p)}}}{\sbb{\Gamma_{(1,\dots,p)}}^{1/\alpha}}\qmand
\label{eq:H}
H_k := \inddd{\wt\ell_k((1,\dots,p)) = (1,\dots,p)},\ k=0,\dots,m.
\eque
Note that $T_k H_k=T_m^* H_k$.
The proof consists of establishing the following asymptotic equivalence in law: 
\begin{align}
\calL\pp{\frac{X_0}{|X_0|},\dots,\frac{X_m}{|X_0|}\mmid |X_0|>x}& \sim \calL\pp{\frac{X_0}{|X_0|},\dots,\frac{X_m}{|X_0|}\mmid |T_m^*|>x, H_0 = 1}\label{eq:1}\\
& \sim \calL\pp{\frac{X_0}{|T_m^*|},\dots,\frac{X_m}{|T_m^*|}\mmid |T_m^*|>x, H_0 = 1}\label{eq:2}\\
& \sim \calL\pp{\frac{T_m^*H_0}{|T_m^*|},\dots,\frac{T_m^*H_m}{|T_m^*|}\mmid |T_m^*|>x, H_0 = 1}\label{eq:3}\\
& \sim \calL\pp{\varepsilon H_0,\dots,\varepsilon H_m\mid H_0 =1}\label{eq:4}.
\end{align}
Here and below, $\calL(\cdots\mid\cdots)$ denote the corresponding conditional law, and $\calL(\cdots)\sim\calL(\cdots)$ means that as $x\to\infty$, the two sides have same limit law,  given that the limit law of one of them exists.    We will leave the verification of first step \eqref{eq:1} to the end. We  first verify steps \eqref{eq:2} and \eqref{eq:3}.  Write, for each $k$,
\[
X_k = T_k + X_k-T_k = T^*_mH_k + T_k(1-H_k)+ X_k-T_k.
\]
Observe that
$
|T_k|(1-H_k)\le w_m^{p/\alpha}\max_{(1,\dots,p)\ne\vvi\in\calD_p}\bb{\Gamma_\vvi}^{-1/\alpha}.
$
In addition, in view of Item (i) before \eqref{eq:p>=1} and independence, we have $\proba(H_k=1)=w_m^{-p}$.
It then follows from   Lemma  \ref{lem:1} and the independence between $T_m^*$ and $H_k$ that 
\equh\label{eq:Xk_tail}
\proba(|X_k|>x) \sim  \proba(|T_m^*H_k|>x)\sim    \frac{w_{m}^{p}\alpha^{p-1} x^{-\alpha}\log^{p-1}x}{p!(p-1)!}\cdot\frac1{w_{m}^{p}} = \frac{\alpha^{p-1} x^{-\alpha} \log^{p-1}(x)}{p!(p-1)!},
\eque
and  
\equh\label{eq:Xk remainder}
\proba(|X_k-T_m^*H_k|>x)=O(x^{-\alpha}\log^{p-2}x)     
\eque 
as $x\to\infty$.
Next, we need the following facts: for random variables $Y$ and $Z$, not necessarily independent, such that $\proba(|Z|>\epsilon x) = o(\proba(|Y|> x))$ as $x\to\infty$ for any $\epsilon>0$, we have as $x\rightarrow\infty$ the convergences in distribution: 
\equh\label{eq:L1}
\calL\pp{\frac {|Y|}{|Y+Z|}\mmid |Y|>x} \to \calL(1) \qmand \calL\pp{\frac {Z}{|Y|} \mmid |Y|>x} \to \calL(0),
\eque
where  $\calL(1)$ and $\calL(0)$ denote the laws of   constants 1 and 0, respectively. 
Indeed, the first relation follows from the second one, whereas the second holds since 
\[\proba(|Z/Y|>\epsilon \ | \ |Y|>x) \le \proba(|Z|>\epsilon x  )/\proba(|Y|>x)\rightarrow 0\] for any $\epsilon>0$. 
Now  \eqref{eq:2} follows by writing for each $k$, 
\[
\frac{X_k}{|X_0|}=\frac{X_k}{|T_m^*| }  \cdot \frac{|T_m^*| }{|T_m^* +X_0-T_m^*| }
\] 
and then applying the first relation in \eqref{eq:L1} with $Y = T_m^*$, $Z = X_0-T_m^*$, as well as  the relations \eqref{eq:Xk_tail} and \eqref{eq:Xk remainder}. The step \eqref{eq:3} follows by writing for each $k$, 
\[
\frac{X_k}{|T_m^*|} = \frac{T_m^*H_k}{|T_m^*|} + \frac{X_k-T_m^*H_k}{|T_m^*|},
\]
and applying the second relation in \eqref{eq:L1}. 

The last step \eqref{eq:4} follows from  
$\sbb{\varepsilon_{(1,\dots,p)}}\EqD\varepsilon$ and independence between $(H_k)_{k=0,\ldots,m}$ and $T_m^*$.
Comparing the definitions \eqref{eq:H} (note that $\wt\ell_k((1,\dots,p)) = (1,\dots,p)$ means $k\in R_{m,i}, i=1,\ldots,p$) and \eqref{eq:Theta}, one readily checks that
\[
\calL(H_0,\dots,H_m\mid H_0 = 1) = \calL(\Theta_0^*,\dots,\Theta_m^*). 
\]

Now we return to verify the first step \eqref{eq:1}. 
Introduce
$B_0(x) := \ccbb{|X_0|>x}, B^*_m(x):=\ccbb{|T_m^*|>x}$,
and   
\[
E = \ccbb{\left(\frac{X_0}{|X_0|},\dots,\frac{X_m}{|X_0|}\right) \in A},
\] 
where $A$ is a Borel set in $\R^{m+1}$ whose boundary is not charged by the limit law \eqref{eq:4}.
It suffices to prove
\equh\label{eq:1'}
\lim_{x\to\infty}\proba\pp{E\mmid |X_0|>x} = \lim_{x\to\infty}\proba\pp{E\mmid |T_m^*|>x, H_0 = 1}.
\eque
It is clear that, from the established steps \eqref{eq:2} to \eqref{eq:4}, the (limit of) right-hand side of \eqref{eq:1'} exists. 
For the left-hand side, write
\begin{align*}
\proba(E\mid |X_0|>x)   
 = \frac{\proba(E\cap \{|X_0|>x,H_0=1\})+\proba(E\cap \{|X_0|>x, H_0=0\})}{\proba(|X_0|>x)}.
\end{align*}
In view of \eqref{eq:top_dominates} and \eqref{eq:remainder_rate}, we have \[\proba(E\cap\{|X_0|>x,H_0 = 0\}) \le \proba(|X_0|>x, H_0 = 0)
=o(\proba(|X_0|>x))
\] as $x\to\infty$. Thereofore, with in addition  \eqref{eq:Xk_tail}, we see that $\proba(E\mid |X_0|>x)$ has the same limit as
\equh\label{eq:fraction}   \frac{\proba(E\cap\{|X_0|>x, H_0=1\})}{\proba(|T_m^*|>x,H_0=1)}
\eque
as $x\to\infty$.  
For the numerator, we  have the following upper and lower bounds:
\begin{align*}
\proba(E\cap\{|X_0|>x,H_0=1\})& \ge \proba(E \cap\{|T_m^*|>(1+\epsilon)x,H_0=1\}) - \proba(|X_0-T_0|>\epsilon x),\\
\proba(E\cap\{|X_0|>x,H_0=1\}) & \le \proba(E\cap\{|T_m^*|>(1-\epsilon)x,H_0=1\}) + \proba(|X_0-T_0|>\epsilon x).
\end{align*}
The limit of \eqref{eq:fraction} with the numerator replaced by the  upper and lower bounds above, as $x\rightarrow\infty$, can be determined by applying  \eqref{eq:remainder_rate}, \eqref{eq:Xk_tail}  and the   relations \eqref{eq:2} to \eqref{eq:4} (with $x$ replaced by $(1\pm\epsilon)x$).  Letting $\epsilon\downarrow0$,  we conclude that \eqref{eq:fraction} has the same limit as
\[
 \frac{\proba(E\cap \{|T_m^*|>x, H_0 = 1\})}{\proba(|T_m^*|>x)\proba(H_0 = 1)}
= \proba(E\mid |T_m^*|>x, H_0 = 1).
\]
This completes the proof of \eqref{eq:1'}.
\end{proof}

\section{Extremal index in the sub-critical regime, $p=2$}\label{sec:p=2}
In this section, we restrict to the sub-critical regime
$p =2, \ \beta\in(0,1/2)$, 
and compute the extremal index directly. 
Recall the series representation of $\{X_k\}_{k=0,\dots,n}$ in \eqref{eq:p>=1}, and for convenience we repeat here 
 \equh\label{eq:representation}
\{X_k\}_{k=0,\ldots,n}\eqd\ccbb{ X_{n,k}}_{k=0,\dots,n} \equiv\ccbb{  w_n^{2/\alpha}\sum_{1\le i_1<i_2}   \frac{\varepsilon_{i_1}\varepsilon_{i_2}}{\Gamma_{i_1}^{1/\alpha}\Gamma_{i_2}^{1/\alpha}}\inddd{k\in   R_{n,i_1}\cap R_{n,i_2}}}_{k=0,\dots,n},
 \eque
 and we follow the same 
 notation  in earlier sections except that we now have $n$ instead of $m$ and will let $n\to\infty$. 
 We also emphasize on the dependence on $n$ of the series representation when writing $X_{n,k}$ instead of $X_k$.
 Introduce
\[
b_n = \pp{\frac14n\log  n}^{1/\alpha}.
\]
It can be verified based on \eqref{eq:top_dominates} that  $P(X_1>b_n)\sim 1/n$ as $n\rightarrow\infty$. Then by a classical extreme   limit theorem (e.g., \cite[Proposition 1.11]{resnick87extreme}, for i.i.d.~copies $\{X\topp 0_k\}_{k\in\N}$ of $X_1$,   we have
\equh\label{eq:iid}
\frac1{b_n}\max_{k=1,\dots,n}X_k\topp 0 \weakto Z_\alpha,
\eque
where $Z_\alpha$ follows a standard $\alpha$-Fr\'echet distribution: $\proba (Z_\alpha\le x) = e^{-x^{-\alpha}}, x\ge 0$.

The goal is to establish a  extreme  limit  theorem  for the model \eqref{eq:representation} in the sub-critical case. 
\begin{Thm}\label{thm:EVT}
Assume $\beta\in(0,1/2)$. 
As $n\rightarrow\infty$, we have
\[
\frac{1}{b_n} \max_{k=1,\ldots,n}  X_k \Rightarrow   \theta^{1/\alpha} Z_\alpha \qmwith \theta = (1-2\beta) \mathfrak q_{F,2},
\]
where $Z_\alpha$ follows a standard $\alpha$-Fr\'echet distribution, and $\mathfrak q_{F,2}$ is as in \eqref{eq:qFp}.
\end{Thm} 
Then, in comparison with \eqref{eq:iid} it follows that $\theta$ is the extremal index of the original process.

 The proof below is different from the one presented in \citep{bai21phase} and a non-trivial adaption is needed to extend the proof here to $p\ge 3$. 
We shall proceed with following approximation procedure.  Fix throughout  a sequence of increasing integers $\{m_n\}_{n\in\N}$ such that 
\begin{equation}\label{eq:m_n rate}
\frac{w_n^2}{n \log^{2/(2-\alpha)}n}  \ll m_n  \ll \frac{w_n^2}{n} 
\end{equation}
as $n\rightarrow\infty$, where $a_n\ll b_n$ means $a_n/b_n\rightarrow 0$ as $n\rightarrow\infty$.  
Introduce
\begin{align*}
\wt{X}_{n,k}:=
 w_n^{2/\alpha}\sum_{1\le i_1<i_2\le m_n}   \frac{\varepsilon_{i_1}\varepsilon_{i_2}}{\Gamma_{i_1}^{1/\alpha}\Gamma_{i_2}^{1/\alpha}}\inddd{k\in   R_{n,i_1}\cap R_{n,i_2}}.
\end{align*}
Then Theorem \ref{thm:EVT} follows from the following two results.
\begin{Prop}\label{prop:EVT main}
As $n\rightarrow\infty$, we have
\[\frac{1}{b_n} \max_{k=1,\ldots,n}  \wt{X}_{n,k} \Rightarrow   \theta^{1/\alpha} Z_\alpha.
\]
\end{Prop}

\begin{Lem}\label{Lem:reduction}
As $n\rightarrow\infty$, we have
\[
\frac{1}{b_n} \max_{k=1,\ldots,n}  |X_{n,k}-\wt{X}_{n,k}|\ConvP 0,
\]
where $\ConvP$ stands for convergence in probability.
\end{Lem}

\subsection{Proof of Proposition \ref{prop:EVT main}}
 Let $\{U_{i}\}_{i\in\N}$ be i.i.d.~uniform random variables independent of everything else. Introduce
\[
V_{i_1,i_2} := \varepsilon_{i_1}\varepsilon_{i_2}  \pp{  U_{i_1} U_{i_2}}^{-1/\alpha},\quad R_{n,i_1,i_2} :=  R_{n,i_1}\cap R_{n,i_2}, \quad i_1,i_2\in \N, \ i_1<i_2.
\]
Then by a well-known relation    $(\Gamma_1,\ldots,\Gamma_{m_n}) \eqd (U_{1,m_n},\ldots, U_{m_n,m_n} )\Gamma_{m_n+1}$, where $U_{1,m_n}<\ldots< U_{m_n,m_n}$ are order statistics of  $\{U_1,\ldots,U_{m_n}\}$, we have
\begin{align}
\ccbb{\wt{X}_{n,k}}_{k=1,\ldots,n}&\eqd \ccbb{
\pp{\frac{w_n}{\Gamma_{m_n+1}}}^{2/\alpha}
 \sum_{1\le i_1<i_2\le m_n}  V_{i_1,i_2}   \inddd{k\in R_{n,i_1,i_2}}}_{k=1,\ldots,n}\notag
\end{align}
Noting that  $\Gamma_{m_n+1}\sim m_n$ as $n\rightarrow\infty$ almost surely, we can instead work with
\begin{align}
 \ccbb{X_{n,k}^*}_{k=1,\ldots,n}= \ccbb{
\pp{\frac{w_n}{m_n}}^{2/\alpha}
 \sum_{1\le i_1<i_2\le m_n}  V_{i_1,i_2}  \inddd{k\in R_{n,i_1,i_2}}}_{k=1,\ldots,n}.  \label{eq:X star}
\end{align}
 
\begin{Lem}
We   have as $n\rightarrow\infty$,
\[
 \rho_n:=\proba\pp{R_{n,1,2} \ne\emptyset} \sim \mathfrak q_{F,2}\frac{n}{w_n^2}.
\]
\end{Lem}
\begin{proof}
Introduce
\begin{align*}
q_{n,1}   &:= \proba\pp{R_{n,1,2}\cap\{1,\dots, n\}\ne\emptyset,\, \max R_{n,1,2}\le n},\\
q_{n,2} & := \proba\pp{R_{n,1,2}\cap\{1,\dots, n\}\ne\emptyset,\, \max R_{n,1,2}> n}.
\end{align*}
Then, 
$q_{n,1}\le \rho_n\le q_{n,1}+q_{n,2}$. 
For $q_{n,1}$ we apply the last-renewal decomposition and Markov property:
\begin{align*}
q_{n,1}  
 & = \summ i1{n} \proba\pp{\max R_{n,1,2} = i\mmid i\in R_{n,1,2}}\proba\pp{i\in R_{n,1,2}} =  \summ i1{n} \mathfrak q_{F,2} \frac1{w_n^2} = \frac{\mathfrak q_{F,2}n}{w_n^2},
\end{align*}
where we have used the fact that conditioning on $i\in R_{n,1,2}$, the count $\sum_{k=i}^\infty 1_{\{k\in R_{n,1,2}\}}$ follows a geometric distribution with mean $\mathfrak q_{F,2}^{-1}$.
For $q_{n,2}$, write first by a similar decomposition based on the last renewal before time $n$, 
\begin{align*}
q_{n,2} & = \summ i1{n}\proba\pp{\max(R_{n,1,2}\cap\{1,\dots,n\}) = i, \max R_{n,1,2}>n} \\
& \le \summ i1{n} \proba\pp{i\in R_{n,1,2}}\proba\pp{\max R_{n,1,2}>n\mmid i\in R_{n,1,2}} \le \summ i1{n} \frac1{w_n^2}\sum_{j=n-i+1}^\infty u(j)^2.
\end{align*}
The last step above
 follows from the renewal property and 
 the union bound.
With $v(i) := \sum_{j=i}^\infty u(j)^2\downarrow 0$ as $i\to\infty$ due to the fact $u(j)\le C j^{\beta-1} $ with $\beta\in (0,1/2)$, the last displayed expression  becomes 
$w_n^{-2}\summ i1{n} v(i) = nw_n^{-2}\summ i1{n}(v(i)/n) = nw_n^{-2}o(1) = o(q_{n,1})$, completing the proof.
\end{proof}

Introduce the following two counting numbers:
\begin{equation}\label{eq:N_n}
N_n :=\sum_{1\le i_1<i_2\le m_n}\inddd{R_{n,i_1,i_2}\cap \{1,\ldots,n\}\neq \emptyset},
\end{equation}
 and
\begin{equation}\label{eq:M_n}
M_n:=\sum_{\substack{1\le i_1<i_2\le m_n, 1\le j_1<j_2\le m_n,\\ \{i_1,i_2\}\neq \{j_1,j_2\}, \{i_1,i_2\}\cap\{j_1,j_2\}\neq \emptyset }} \inddd{ R_{n,i_1,i_2}\cap \{1,\ldots,n\}\neq \emptyset,\, R_{n,j_1,j_2}\cap \{1,\ldots,n\} \neq \emptyset}.
\end{equation}
\begin{Lem}\label{Lem:N M estimate}
We have   as $n\rightarrow\infty$ that
\begin{equation}\label{eq:lambda}
\esp N_n \sim   \lambda_n:= \frac{
\mathfrak q_{F,2}
}{2}\frac{n m_n^2}{w_n^2}   \rightarrow\infty , \quad  \frac{N_n}{\lambda_n}\ConvP   1  
\end{equation}
and
$\esp M_n\le C   \frac{n^2 m_n^3}{w_n^4} =o(\lambda_n)$.
\end{Lem}
\begin{proof}
For the first part, we have
\[
\esp N_n \sim \frac{m_n^2 \rho_n}{2}=\frac{
\mathfrak q_{F,2}
}{2}\frac{n m_n^2}{w_n^2},
\]
which tends to $\infty$ due to $m_n\gg w_n^2/(n\log^{2/(2-\alpha)}n)$, the first part of assumption \eqref{eq:m_n rate}.
Next, note that
\begin{align*}
\rho'_n&:=\proba\pp{  R_{n,1,2}\cap \{1,\ldots,n\}\neq \emptyset,\, R_{n,1,3}\cap \{1,\ldots,n\} \neq \emptyset  }\\
& \le  
\summ k1n\summ{k'}1n
 \proba\pp{k\in R_{n,1,2}, k'\in  R_{n,1,3}}
 \le  \frac{2 }{w_n^3}\summ k1n \summ{k'}kn   u(k'-k) \le C \frac{n^2}{w_n^4}.
\end{align*}
Hence
\begin{align*}
\esp M_n \le  C  \rho_n' m_n^3  \le C   \frac{n^2 m_n^3}{w_n^4}  =o(\lambda_n),
\end{align*}
where the last relation   follows from $m_n\ll w_n^2/n$, the second part of \eqref{eq:m_n rate}.

Next,  by a decomposition of  the  double sum over $1\le i_1<i_2\le m_n$ and $1\le j_1<j_2\le m_n$ according to $|\{i_1,i_2\}\cap \{j_1,j_2\}|=0,1,2$, we have as $n\rightarrow\infty$,
\[
\esp N_n^2= {m_n \choose 2} {m_n-2\choose 2} \rho_n^2+ \esp M_n+ \esp N_n=(\esp N_n)^2 (1-o(1)) + O(\lambda_n),
\]
and hence $\var (N_n) =o((\esp N_n)^2)$, 
which concludes the convergence in probability in \eqref{eq:lambda}.
\end{proof}

\begin{proof}[Proof of Proposition \ref{prop:EVT main}]
We shall work with ${X}_{n,k}^*$ in \eqref{eq:X star}. The key underlying structure is that its partial maximum can be approximated by a collection of $\lambda_n$ (see \eqref{eq:lambda}) i.i.d.~random variables, as summarized in \eqref{eq:EVT main} below, and this approximation alone explains why the extra factor $\mathsf D_{2,\beta} = 1-2\beta$ shows up in the extremal index compared to the candidate extremal index. 

We start by setting the random index set 
\[
J_n(k):= \{(i_1,i_2)\in \{1,\ldots,m_n\}^2:\  i_1<i_2, \  k\in R_{n,i_1,i_2} \}, \quad k=1,\ldots,n.
\]
So 
\equh\label{eq:max Xstar rewrite}
 \max_{k=1,\ldots,n}   {X}_{n,k}^* 
=  \pp{\frac{w_n^2}{ m_n^2}}^{1/\alpha}  \max_{k=1,\ldots,n} \left(\sum_{(i_1,i_2)\in J_n(k)} V_{i_1,i_2}\right)   ,
\eque
where the sum  over $(i_1,i_2)\in J_n(k)$    is understood as $0$ if $J_n(k)=\emptyset$.  
Define also 
\[I_n(k):=\bigcup_{(i_1,i_2)\in J_n(k)}\{ i_1,i_2 \}.\] 
 In fact $|J_n(k)|={|I_n(k)|\choose 2}$, and hence $|J_n(k)|$ cannot take arbitrary integer values.
  But for simplicity of notation we shall still write a consecutive integer range for $|J_n(k)|$ below.
Let \begin{align*}
\mathcal K_n:=\Big\{&k\in \{1,\ldots,n\}:  |J_n(k)|=1,~\text{and} \\  &\text{  either }  J_n(k')=J_n(k) \text{ or } I_n(k')\cap I_n(k) =\emptyset \, \ \forall k'\in \{1,\ldots,n\}\setminus \{k\}    \Big\},
\end{align*}
and set $\mathcal K_n^c:= \{1,\ldots,n\}\setminus \mathcal K_n$.  Then by independence,
\begin{equation}\label{eq:max single}
\max_{k\in \mathcal K_n} \left(\sum_{(i_1,i_2)\in J_n(k)} V_{i_1,i_2}\right)\eqd   \max_{\ell = 1,\dots, \wt{N}_n} V_{\ell} \qmwith \wt{N}_n:=|\cup_{k\in \mathcal{K}_n} J_n(k)|,
\end{equation}
where $\{V_\ell\}_{\ell\in\N}$ are i.i.d.\ of the same distribution as $V_{1,2}$ and independent of everything else. 
Here and below, when a maximum is performed over an empty index set, it is understood as $0$.

Next, observe that that $\wt{N}_n\le N_n$ with $N_n$ in \eqref{eq:N_n}. In addition,   any $(i_1,i_2)$ satisfying $R_{n,i_1,i_2}\neq \emptyset$ and $1\le i_1<i_2\le m_n$,  but not included in $\cup_{k\in \mathcal{K}_n} J_n(k)$,  must have been counted at least once by $M_n$ in \eqref{eq:M_n}.  So  
$0\le N_n-\wt{N}_n\le M_n$.
Recall    $\lambda_n$ in \eqref{eq:lambda}. 

In summary, one can prove
\begin{align}
\limn\proba\pp{\frac{1}{b_n} \max_{k\in \mathcal K_n} {X}_{n,k}^*\le x} & = \limn \proba\pp{ \pp{\frac{w_n^2}{b_n^{\alpha} m_n^2}}^{1/\alpha}\max_{\ell=1,\dots, \wt{N}_n} V_{\ell}\le x}\nonumber\\
&= \limn\proba\pp{\pp{\frac{w_n^2}{b_n^{\alpha} m_n^2}}^{1/\alpha}   \max_{ \ell=1,\dots, \floor{\lambda_n}} V_{\ell}\le x}\label{eq:EVT main} \\
& = \proba\pp{\theta^{1/\alpha} Z_\alpha \le x}, \mfa x>0. \label{eq:EVT std}
\end{align}
Indeed, the first relation above follows from \eqref{eq:max Xstar rewrite} and \eqref{eq:max single}. We postpone the proof for \eqref{eq:EVT main} for a moment. 
 Then \eqref{eq:EVT std} follows from classical extreme-value limit theorem for $\lfloor \lambda_n \rfloor$ i.i.d.\
 random variables of regularly-varying tail (e.g., \cite[Proposition 1.11]{resnick87extreme}). 
 The normalization, say $d_n$, for 
 \[
 \frac1{d_n} \max_{\ell=1,\dots,\floor{\lambda_n}} V_\ell\weakto Z_\alpha,
 \] is well-known to be determined by $\limn\lambda_n\proba(V_1>d_n) = 1$. 
  It is elementary to verify that
\begin{equation}\label{eq:V marginal tail}
\proba(V_1<-x)=\proba(V_1>x)=\frac{1}{2}\proba(U_1U_2<x^{-\alpha})\sim \frac{\alpha}{2} x^{-\alpha} \log x
\end{equation}
as $x\rightarrow\infty$.
Recall also $\lambda_n \sim (\mathfrak q_{F,2}/2)nm_n^2/w_n^2$ and the   requirement \eqref{eq:m_n rate} which implies 
 $\log(\lambda_n)\sim \log(w_n^2/n)\sim (1-2\beta)\log n$. So we have
 \begin{align*}
 d_n   & \sim \pp{\frac1{2} \lambda_n\log \lambda_n}^{1/\alpha} \sim \pp{\frac{1-2\beta}2\lambda_n\log n}^{1/\alpha}\\
 &\sim \pp{\frac{b_n^\alpha m_n^2}{w_n^2}}^{1/\alpha}\pp{{\mathfrak q_{F,2}(1-2\beta)}}^{1/\alpha} = \pp{\frac{b_n^\alpha m_n^2}{w_n^2}}^{1/\alpha}\theta^{1/\alpha},
 \end{align*}
 yielding \eqref{eq:EVT std}. 
Now we check \eqref{eq:EVT main}. 
Introduce $f(x) := ((1/2)x\log x)^{1/\alpha}$. So $d_n \sim f(\lambda_n)$. Introduce accordingly $\what d_n := f(\wt N_n) =  ((1/2)\wt N_n\log\wt N_n)^{1/\alpha}$.
By Lemma \ref{Lem:N M estimate}, we have $\wt{N}_n/\lambda_n\ConvP 1$ and $\lambda_n\rightarrow\infty$ as $n\rightarrow\infty$, 
 and it follows that $\what d_n/d_n\ConvP 1$ since $f$ is a regularly varying function (e.g.~\citep[Theorem 1.12]{kulik20heavy}). Write $Z_n = f(n)\inv \max_{\ell = 1,\dots,n}V_\ell$ and $d_n \inv \max_{\ell=1,\dots,\wt N_n}V_\ell = (\what d_n/d_n) \cdot Z_{\wt N_n}$. So to obtain \eqref{eq:EVT main} it is remains to argue
$Z_{\wt N_n}$ and $Z_n$  have the same limit distribution as $n\to\infty$. The last step is an exercise.

To complete the proof, it remains to show that  as $n\rightarrow\infty$, $b_n\inv \max_{k\in \mathcal K_n^c}   {X}_{n,k}^*\ConvP 0$. 
We fix $p^*$ large enough so that $1-(p^*+1)\beta<0$ (recall $\beta\in(0,1/2)$). Then 
\begin{align*}
\proba(|J_n(k)|> p^* \text{ for some }k =1,\ldots,n)
& \le  \sum_{k=1}^n  \binom{m_n}{p^*+1}w_n^{-p^*-1} 
\le    C \frac{w_n^{p^*+1}}{n^{p^*}}\le C n^{1-(p^*+1)\beta}\rightarrow 0,
\end{align*}
where we have used \eqref{eq:m_n rate} in the second inequality above.
Hence with probability tending to 1 as $n\rightarrow\infty$, we have  
\begin{align*}
\frac{1}{b_n} \max_{k\in \mathcal K_n^c}   {X}_{n,k}^* =  \pp{\frac{w_n^2}{b_n^{\alpha} m_n^2}}^{1/\alpha}  \max_{j=1,\ldots,p^*} \max_{
k\in \mathcal K_n^c
} \sum_{(i_1,i_2)\in J_n(k),\ |J_n(k)|=j} V_{i_1,i_2}.
\end{align*}
So it suffices to show for fixed $j=1,\ldots,p^*$ as $n\rightarrow\infty$ that 
\begin{equation}\label{eq:goal id}
\pp{\frac{w_n^2}{b_n^{\alpha} m_n^2}}^{1/\alpha} \max_{k\in\mathcal K_n^c
} W_j(k) \ConvP 0 \qmwith
 W_j(k)=\sum_{(i_1,i_2)\in J_n(k),\ |J_n(k)|=j} |V_{i_1,i_2}|.
 \end{equation}
 Note that $\{W_j(k)\}_{k\in\mathcal K_n^c}$ are, 
 given $\mathcal K_n^c$, 
 identically distributed but possibly dependent random variables, and the total number of distinct $W_j(k)$'s, say $K_j(n)$, does not exceed $N_n-\wt{N}_n\le M_n$.   On the other hand in view of \eqref{eq:V marginal tail}, 
\begin{equation}\label{eq:tail bound W}
\proba(W_j(1)>x)\le j \proba(|V_{1,2}|>x/j)\le C x^{-\alpha} \log x
\end{equation} 
for all $x>0$ and some constant $C>0$. Suppose $\{\wt{W}_j(\ell)\}_{\ell\in \N}$ are i.i.d.\ copies of $W_j(1)$ and independent of everything else.  Then by the tail bound  \eqref{eq:tail bound W}, the tail estimate \eqref{eq:V marginal tail},  the extreme-value limit theorem for $|V_\ell|$ and the fact that $M_n/\lambda_n\ConvP 0$ (Lemma \ref{Lem:N M estimate}), we have   
\[
 \proba\pp{ \pp{\frac{w_n^2}{b_n^{\alpha} m_n^2}}^{1/\alpha} \max_{\ell=1,\ldots,K_j(n)} \wt{W}_j(\ell)>\epsilon}\le  C \proba\pp{ \pp{\frac{w_n^2}{b_n^{\alpha} m_n^2}}^{1/\alpha} \max_{\ell=1,\ldots,M_n} |V_\ell|>\epsilon}\rightarrow 0
\]
for any $\epsilon>0$ as $n\rightarrow\infty$.
Then \eqref{eq:goal id} follows  from     \cite[Proposition 9.7.3]{samorodnitsky16stochastic}.
\end{proof}

\subsection{Proof of Lemma \ref{Lem:reduction}}
Lemma \ref{Lem:reduction} follows if one establishes that
\begin{align*}
A_n & := \frac{w_n^{2/\alpha}}{b_n}\max_{k=1,\ldots,n}\sum_{1\le i_1 \le m_n<i_2}   \frac{\varepsilon_{i_1}\varepsilon_{i_2}}{\Gamma_{i_1}^{1/\alpha}\Gamma_{i_2}^{1/\alpha}}\inddd{k\in   R_{n,i_1,i_2}} \ConvP 0,
\\
B_n& :=  \frac{w_n^{2/\alpha}}{b_n} \max_{k=1,\ldots,n}\sum_{  m_n<i_1<i_2 }   \frac{\varepsilon_{i_1}\varepsilon_{i_2}}{\Gamma_{i_1}^{1/\alpha}\Gamma_{i_2}^{1/\alpha}}\inddd{k\in   R_{n,i_1,i_2}} \ConvP 0,
\end{align*}
as $n\rightarrow\infty$. 
First, fix a number $m>4/\alpha$. Assume without loss of generality that $m_n>4/\alpha$.  Then write
\[
A_n=A_{n,1}+A_{n,2}:=  \frac{w_n^{2/\alpha}}{b_n} \max_{k=1,\ldots,n}\sum_{m \le i_1\le m_n<i_2}  \cdots +   \frac{w_n^{2/\alpha}}{b_n}\max_{k=1,\ldots,n} \sum_{1 \le i_1<m,  i_2>m_n} \cdots.
\]
Then by bounding the maximum of non-negative numbers with their sum, the orthogonality induced by $\epsilon_{i_1}\epsilon_{i_2}$ and \citep[the inequality (3.2)]{samorodnitsky89asymptotic}, we have
\begin{align*}
\esp A_{n,1}^2 & \le  \frac{ w_n^{4/\alpha}}{b_n^2} \esp\sum_{k=1}^n \pp{\sum_{m \le i_1\le m_n<i_2} \frac{\varepsilon_{i_1}\varepsilon_{i_2}}{\Gamma_{i_1}^{1/\alpha}\Gamma_{i_2}^{1/\alpha}}  1_{\{k\in R_{n,i_1,i_2}\}}}^2 \\
&=\frac{ w_n^{4/\alpha}}{b_n^2} \sum_{m \le i_1\le m_n<i_2}  \esp\pp{\Gamma_{i_1}^{-2/\alpha} \Gamma_{i_2}^{-2/\alpha}}   \pp{\sum_{k=1}^n \proba\pp{k\in R_{n,i_1,i_2}}}
 \\&\le C \frac{n w_n^{4/\alpha-2}}{b_n^2}  \sum_{m \le i_1\le m_n<i_2}   i_1^{-2/\alpha}  i_2^{-2/\alpha}   
\le  C  \frac{n w_n^{4/\alpha-2}}{b_n^2} m_n^{1-2/\alpha} \rightarrow 0,
\end{align*}
where we have used the relation $\frac{w_n^2}{n \log^{2/(2-\alpha)}(n)}  \ll m_n $ in the first part of \eqref{eq:m_n rate} in the last step.

For $A_{n,2}$, since $i_1$   takes finitely many values, it suffices to show for fixed $i_1$  that
\begin{align*}
b_n^{-2} w_n^{4/\alpha} \esp \left| \max_{k=1,\ldots,n} \sum_{i_2>m_n} \Gamma_{i_2}^{-1/\alpha}   \epsilon_{i_2}    1_{\{k\in R_{n,i_1,i_2}\}} \right|^2  \le C  n b_n^{-2} w_n^{4/\alpha-2} m_n^{1-2/\alpha}\rightarrow 0,
\end{align*}
which follows as above.
For $B_n$, similarly we have
\[
\esp B_n^2\le C n b_n^{-2} w_n^{4/\alpha-2}  \sum_{m_n<  i_1<i_2}   i_1^{-2/\alpha}  i_2^{-2/\alpha} \le C  n b_n^{-2} w_n^{4/\alpha-2} m_n^{2-4/\alpha}  \rightarrow 0.
\]

\section{Anti-clustering condition when $\beta_p<0$}\label{sec:anti}
Assume $\beta_p<0$ in this section. Based on the convergence of the tail processes we can actually prove the convergence of the single-block cluster point process, which is based on the verification of the  anti-clustering condition \eqref{eq:AC}. 
Strictly speaking we prove a stronger condition known as the $\calS(r_n,a_n)$ condition as in \citep[P.243]{kulik20heavy}, in Lemma \ref{lem:AC} below. This condition could have other consequences, notably when proving limit theorems for tail empirical processes \citep[Chapter 9]{kulik20heavy}. It remains an interesting question whether one could prove limit theorems for tail empirical processes for our model here: the classical approach relies also on certain mixing-type condition and   does not seem   applicable here. 
\begin{Prop}
Assume $\beta_p<0$. For $a_n\to \infty$, $r_n\to\infty$ and $r_n = o(\log a_n)$, with $M_{r_n}:=\max_{k=1,\dots,r_n}|X_k|$, the limit of
\[
\calL\pp{\summ i1{r_n}\delta_{X_i/a_n}\mmid M_{r_n}>a_n}
\]
in the space of $\mathfrak M_p(\wb\R\setminus\{0\})$ is $\mathfrak G\delta_\varepsilon$, a unit mass at a Rademacher random variable $\varepsilon$ multiplied by a geometric random variable $\mathfrak G$ with mean $1/\mathfrak q_{F,p}$  and independent from $\varepsilon$. 
\end{Prop}
\begin{proof}
By Theorem \ref{thm:tail}, \cite[Theorem 4.3]{basrak09regularly} and a continuous mapping argument (cf.\ \citep[Remark 4.6]{basrak09regularly}),  the limit is a point measure
\[
\sif i0 \delta_{\varepsilon\Theta_i^*}
\]
restricted to $\mathfrak M_p(\wb \R\setminus\{0\})$
given that the anti-clustering condition \eqref{eq:AC}  holds \citep[e.g.~Condition 4.1]{basrak09regularly}, which is verified in Lemma \ref{lem:AC} below.  
Note that because of the restriction, the point process above is actually 
$\spp{\sif i0 \inddd{\Theta_i^* = 1}}\delta_{\varepsilon }$
(note that $\vv \Theta_0^*= 1$ always),
and the summation in the parenthesis is readily checked to be a geometric random variable with the desired law, by examining the definition of $\vv\Theta^*$ in \eqref{eq:Theta*} and the renewal property. 
\end{proof}
The following lemma implies the anti-clustering condition \eqref{eq:AC} by a simple argument based on union bound and the stationarity of the sequence.
\begin{Lem}[Condition $\calS(r_n,a_n)$]\label{lem:AC}
Assume $\beta_p<0$.
For $a_n\to\infty$, $r_n\to\infty$ and $r_n = o(\log a_n)$, 
\equh\label{eq:union}
\proba\pp{\max_{\ell\le |k|\le r_n}|X_k|>a_n\eta\mmid |X_0|>a_n\eta} \le \frac2{\proba(|X_0|>a_n\eta)}\sum_{k=\ell}^{r_n}\proba(|X_0|>a_n\eta, |X_{k}|>a_n\eta)\rightarrow 0.
\eque
\end{Lem}
\begin{proof}
Unlike in the proof of Theorem \ref{thm:tail}, we only need a series representation for two-dimensional  $(X_0,X_K)$, for each $K\in\N$ fixed. This can be done in the same way as \eqref{eq:p>=1} is derived, with the only modification that the restriction $\vv\tau^*\cap\{0,\dots,m\}\ne\emptyset$ is now replaced by $\vv\tau^*\cap\{0,K\}\ne\emptyset$. More precisely, let
$\{\wt R_{K,i}\}_{i\in\N}$ be i.i.d.~copies of $\wt R_K$, of which the law, denoted by $\wt \mu_K$, is determined by
\[
\frac{d\wt \mu_K}{d\mu^*} = \frac{\inddd{\vv\tau^*\cap\{0,K\}\ne\emptyset}}{\mu^*(\{\vv\tau^*:\vv\tau^*\cap \{0,K\}\ne\emptyset\})} = \frac{\inddd{\vv\tau^*\cap\{0,K\}\ne\emptyset}}{\wt w_K},
\]
and one can compute by inclusion-exclusion formula that $\wt w_K =2-u(K)$.
Then we arrive at
\[
\ccbb{X_k}_{k=0,K}\eqd\ccbb{\wt w_K^{p/\alpha}\sum_{\vvi\in\calD_p}
\frac{[\varepsilon_{  \vvi }]}{[\Gamma_{  \vvi   }]^{1/\alpha}}\inddd{k\in \wt R_{K,i}}}_{k=0,K},
\] 
where as usual $\{\wt R_{K,i}\}_{i\in\N}$ are independent from $\{\varepsilon_i\}_{i\in\N}$ and $\{\Gamma_i\}_{i\in\N}$. 

This time, introduce (similar to \eqref{eq:ell_k})
for $k=0,K$, 
\begin{align*}
\what\ell_{k}(1) & :=\min\ccbb{j\in\N:k\in \wt R_{K,j}},\nonumber\\
 \what\ell_{k}(s) &:=\min\ccbb{j>\what \ell_{K,k}(s-1):k\in  \wt R_{K,j}},s\ge 2.
\end{align*}
We can write
\equh\label{eq:what Tk}
\what X_k:= \wt w_K^{p/\alpha}\sum_{\vvi\in\calD_p}
\frac{[\varepsilon_{\what\ell_{k}(\vvi)}]}{[\Gamma_{\what\ell_{k}(\vvi)}]^{1/\alpha}} \qmwith
\what T_k := \wt w_K^{p/\alpha} \frac{[\varepsilon_{\what\ell_{k}((1,\dots,p))}]}{[\Gamma_{\what\ell_{k}((1,\dots,p))}]^{1/\alpha}}, \ k=0,K.
\eque
Then,  $(\what X_0,\what X_K)\eqd(X_0,X_K)$ and $(\what X_k, \what T_k)\eqd(X_k,T_k)$, $k=0,K$. 
Then  
\begin{align}
\proba(|X_0|>a_n\eta,|X_K|>a_n\eta) & \le \proba\pp{|\what T_{0}|>a_n\eta/2, |\what T_{K}|>a_n\eta/2} + 2\proba\pp{|\what X_{0}-\what T_{0}|>a_n\eta/2}\nonumber\\
& \le \proba\pp{|\what T_{0}|>a_n\eta/2, |\what T_{K}|>a_n\eta/2} + C \mathsf q_{p-1}\pp{\pp{a_n\eta/2}^\alpha},\label{eq:AC1}
\end{align}
where the last step follows from \eqref{eq:remainder_rate}.
Introduce
\[
\what H_{k} := \inddd{\what \ell_{k}((1,\dots,p)) = (1,\dots,p)}, \quad k=0,K.
\]
Therefore we have, using the representation \eqref{eq:what Tk} again,
\begin{align}
\proba & \pp{|\what T_{0}|>\frac{a_n\eta}2, |\what T_{K}|>\frac{a_n\eta}2}\nonumber
\\
& \quad\le \proba\pp{\frac{(2-u(K))^{p/\alpha}}{[\Gamma_{(1,\dots,p)}]^{1/\alpha}}>\frac{a_n\eta}2}\proba\pp{\what H_{0} = \what H_{K} = 1}  + 
\proba\pp{\frac{(2-u(K))^{p/\alpha}}{\pp{\Gamma_1\cdots\Gamma_{p-1}\Gamma_{p+1}}^{1/\alpha}}>\frac{a_n\eta}2}\nonumber\\
& \quad\label{eq:AC2}
\le C\mathsf q_p((a_n\eta)^\alpha/2^{p+\alpha})\proba\pp{\what H_{0} = \what H_{K} = 1} + C\mathsf q_{p-1}((a_n\eta)^\alpha/2^{p+\alpha}). 
\end{align}
It is straightforward to compute
\begin{align}
\proba\pp{\what H_{0} = \what H_{K} = 1} & = \proba\pp{\what H_{0} = 1}\proba\pp{\what H_{K} = 1\mmid \what H_{0} = 1} \le \proba\pp{\what H_{K} = 1\mmid \what H_{0} = 1}
= u(K)^p.\label{eq:AC3}
\end{align}
To sum up, by \eqref{eq:AC1}, \eqref{eq:AC2} and \eqref{eq:AC3}, the right-hand side of \eqref{eq:union} is further bounded from above by
\[
\frac{C\mathsf q_p((a_n\eta)^\alpha/2^{p+\alpha})}{\mathsf q_p((a_n\eta)^\alpha)}\sif k\ell u(k)^{p} + \frac{Cr_n\mathsf q_{p-1}((a_n\eta)^\alpha/2^{p+\alpha})}{\mathsf q_p((a_n\eta)^\alpha)} \le C\sif k\ell u(k)^{-p}+C\frac {r_n}{\log a_n}. 
\]
Recall $u(k)$ in \eqref{eq:u}. Since $\beta_p<0$, $\sif k1 u(k)^{p}<\infty$. We have thus proved the desired result.
\end{proof}

\begin{acks}[Acknowledgments]
The authors are grateful  to  Gennady Samorodnitsky for suggesting to investigate the tail processes for stable-regenerative multiple-stable processes and for several helpful discussions, and to Rafa\l~Kulik for very detailed explanations regarding tail processes and extremal indices, a very careful reading of a primitive version of the paper, and stimulating discussions. The authors would also like to thank Olivier Wintenberger for pointing out the reference \citep{smith88counterexample}.  The authors thank two anonymous referees for careful reading and constructive comments.
\end{acks}

\begin{funding}
%
The second author was supported in part by 
Army Research Office, USA (W911NF-20-1-0139).
\end{funding}

\bibliographystyle{imsart-nameyear} 

\bibliography{../../include/references,../../include/references18}
\end{document}